\documentclass[11pt]{amsart}
\usepackage{amssymb}
\usepackage{graphics}
\usepackage{latexsym}
\usepackage{amsmath}
\usepackage{amssymb,amsthm,amsfonts}
\usepackage{amscd}
\usepackage[arrow, matrix, curve]{xy}
\usepackage{syntonly}
\title[Mass formula]{Deuring's mass formula of a Mumford family}
\author[Mao Sheng]{Mao Sheng}
\email{msheng@ustc.edu.cn} \address{School of Mathematical Sciences,
University of Science and Technology of China, Hefei, 230026, China}
\author[Kang Zuo]{Kang Zuo}
\email{zuok@uni-mainz.de}\address{Institut f\"{u}r  Mathematik,
Universit\"{a}t Mainz, Mainz, 55099, Germany}

\begin{document}
\theoremstyle{plain}
\newtheorem{thm}{Theorem}[section]
\newtheorem{theorem}[thm]{Theorem}
\newtheorem{lemma}[thm]{Lemma}
\newtheorem{corollary}[thm]{Corollary}
\newtheorem{proposition}[thm]{Proposition}
\newtheorem{addendum}[thm]{Addendum}
\newtheorem{variant}[thm]{Variant}
\theoremstyle{definition}
\newtheorem{lemma and definition}[thm]{Lemma and Definition}
\newtheorem{construction}[thm]{Construction}
\newtheorem{notations}[thm]{Notations}
\newtheorem{question}[thm]{Question}
\newtheorem{problem}[thm]{Problem}
\newtheorem{remark}[thm]{Remark}
\newtheorem{remarks}[thm]{Remarks}
\newtheorem{definition}[thm]{Definition}
\newtheorem{claim}[thm]{Claim}
\newtheorem{assumption}[thm]{Assumption}
\newtheorem{assumptions}[thm]{Assumptions}
\newtheorem{properties}[thm]{Properties}
\newtheorem{example}[thm]{Example}
\newtheorem{conjecture}[thm]{Conjecture}
\newtheorem{proposition and definition}[thm]{Proposition and Definition}
\numberwithin{equation}{thm}

\newcommand{\pP}{{\mathfrak p}}
\newcommand{\sA}{{\mathcal A}}
\newcommand{\sB}{{\mathcal B}}
\newcommand{\sC}{{\mathcal C}}
\newcommand{\sD}{{\mathcal D}}
\newcommand{\sE}{{\mathcal E}}
\newcommand{\sF}{{\mathcal F}}
\newcommand{\sG}{{\mathcal G}}
\newcommand{\sH}{{\mathcal H}}
\newcommand{\sI}{{\mathcal I}}
\newcommand{\sJ}{{\mathcal J}}
\newcommand{\sK}{{\mathcal K}}
\newcommand{\sL}{{\mathcal L}}
\newcommand{\sM}{{\mathcal M}}
\newcommand{\sN}{{\mathcal N}}
\newcommand{\sO}{{\mathcal O}}
\newcommand{\sP}{{\mathcal P}}
\newcommand{\sQ}{{\mathcal Q}}
\newcommand{\sR}{{\mathcal R}}
\newcommand{\sS}{{\mathcal S}}
\newcommand{\sT}{{\mathcal T}}
\newcommand{\sU}{{\mathcal U}}
\newcommand{\sV}{{\mathcal V}}
\newcommand{\sW}{{\mathcal W}}
\newcommand{\sX}{{\mathcal X}}
\newcommand{\sY}{{\mathcal Y}}
\newcommand{\sZ}{{\mathcal Z}}
\newcommand{\A}{{\mathbb A}}
\newcommand{\B}{{\mathbb B}}
\newcommand{\C}{{\mathbb C}}
\newcommand{\D}{{\mathbb D}}
\newcommand{\E}{{\mathbb E}}
\newcommand{\F}{{\mathbb F}}
\newcommand{\G}{{\mathbb G}}
\newcommand{\HH}{{\mathbb H}}
\newcommand{\I}{{\mathbb I}}
\newcommand{\J}{{\mathbb J}}
\renewcommand{\L}{{\mathbb L}}
\newcommand{\M}{{\mathbb M}}
\newcommand{\N}{{\mathbb N}}
\renewcommand{\P}{{\mathbb P}}
\newcommand{\Q}{{\mathbb Q}}
\newcommand{\R}{{\mathbb R}}
\newcommand{\SSS}{{\mathbb S}}
\newcommand{\T}{{\mathbb T}}
\newcommand{\U}{{\mathbb U}}
\newcommand{\V}{{\mathbb V}}
\newcommand{\W}{{\mathbb W}}
\newcommand{\X}{{\mathbb X}}
\newcommand{\Y}{{\mathbb Y}}
\newcommand{\Z}{{\mathbb Z}}
\newcommand{\id}{{\rm id}}
\newcommand{\rank}{{\rm rank}}
\newcommand{\END}{{\mathbb E}{\rm nd}}
\newcommand{\End}{{\rm End}}
\newcommand{\Hom}{{\rm Hom}}
\newcommand{\Hg}{{\rm Hg}}
\newcommand{\tr}{{\rm tr}}
\newcommand{\Cor}{{\rm Cor}}
\newcommand{\GL}{\mathrm{GL}}
\newcommand{\SL}{\mathrm{SL}}
\newcommand{\Aut}{\mathrm{Aut}}
\newcommand{\Sym}{\mathrm{Sym}}
\newcommand{\DD}{\mathbf{D}}
\newcommand{\EE}{\mathbf{E}}
\newcommand{\Gal}{\mathrm{Gal}}
\newcommand{\GSp}{\mathrm{GSp}}
\newcommand{\Spf}{\mathrm{Spf}}
\newcommand{\Spec}{\mathrm{Spec}}
\newcommand{\SU}{\mathrm{SU}}
\newcommand{\Res}{\mathrm{Res}}
\newcommand{\Rep}{\mathrm{Rep}}
\thanks{This work is partially supported by the SFB/TR 45 `Periods, Moduli
Spaces and Arithmetic of Algebraic Varieties' of the DFG.}

\begin{abstract}
We study the Newton polygon jumping locus of a Mumford family in
char $p$. Our main result says that, under a mild assumption on $p$,
the jumping locus consists of only supersingular points and its
cardinality is equal to $(p^r-1)(g-1)$, where $r$ is the degree of
the defining field of the base curve of a Mumford family in char $p$
and $g$ is the genus of the curve. The underlying technique is the
$p$-adic Hodge theory.
\end{abstract}

\maketitle

\section{Introduction}
Let $k$ be a finite field of char $p$ and $\bar k$ an algebraic
closure of $k$. A basic result of M. Deuring \cite{Deu} says that
elliptic curves over $\bar k$ can be divided into two classes:
ordinary and supersingular, and there are finitely many
supersingular elliptic curves up to isomorphisms. The mass formula
is a formula on the number of isomorphism classes of supersingular
elliptic curves. For an odd prime $p$, it can be deduced from the
fact that there are exactly $\frac{p-1}{2}$ supersingular elliptic
curves in the Legendre family
$$y^2=x(x-1)(x-t),\quad t\neq 0,1.$$
The purpose of this paper is to give analogous results for a Mumford
family. In \cite{Mu}, D. Mumford gave the \emph{first} example of families of Hodge type, which is characterized by the Hodge group (called also the special Mumford-Tate group in the literature) but not by the endomorphism algebra. We briefly recall the construction as follows. Let $F$ be a totally real cubic field with three real places $\tau_1,\tau_2,\tau_3$, and $D$ a quaternion division algebra over $F$ such that $D$ splits at one real place of $F$ and its corestriction to $\Q$ splits, i.e.,
$$
\Cor_{F|\Q}D:=(D^{(1)}\otimes D^{(2)}\otimes D^{(3)})^{\Gal(\bar{\Q}|\Q)}\cong M_{8}(\Q),
$$
where $D^{(i)}:=D\otimes_{F,\tau_i}\bar{\Q}$. It gives rise to families of abelian four-folds over smooth projective arithmetic quotients of the upper half plane,
whose general fiber has only $\Z$ as its endomorphism ring. For some purpose, his construction has been generalized (and also characterized) in the work of Viehweg and the second named author (see \cite{VZ}, particularly Theorem 0.5).\\

Now let $F$ be a totally real field of degree $d\geq 3$, whose
ring of algebraic integers is denoted by $\sO$, and $D$ a quaternion
division algebra over $F$, which is split only at one real place of
$F$. The corestriction $\Cor_{F|\Q}D$ is a central simple
$\Q$-algebra. Following the construction of Mumford loc. cit., one is able to associate $\Cor_{F|\Q}D$ with a Shimura
curve of Hodge type (see \S\ref{Section on Shimura curve of Hodge type} for details). The universal family of abelian varieties over
such a Shimura curve with a suitable level structure is called a
Mumford family in this paper. In order to do reduction modulo $p$, we need also a
natural integral model of a Mumford family. For that, we make the
following
\begin{assumption}\label{assumption on p from integral model}
Assume $p\geq 3$ and does not divide the discriminants of $F$ and
$D$.
\end{assumption}
After the work of M. Kisin \cite{Ki}, one is able to define the
integral canonical model of the Shimura curve over any prime
$\mathfrak p$ of $F$ over $p$ together with a universal abelian
scheme over the integral model, which is defined over
$\sO_{(\mathfrak p)}$. Fix such a universal abelian scheme and
denote its completion at $\mathfrak p$ by $f: X\to M$, whose modulo
$\mathfrak p$ reduction is denoted by $f_0:X_0\to M_0$. By the
theorem of Grothendieck-Katz, the Newton polygon jumping locus
$\sS\subset M_0(\bar k)$ consists of finitely many points. Our main
result is stated as follows:
\begin{theorem}[Theorem \ref{classification and existence of newton polygon}, Corollary \ref{Newton jumping locus}]\label{main result}
The Newton jumping locus $\sS$ consists only of supersingular
points. Assume $p\geq 5$ additionally, then one has a mass formula
for the cardinality of $\sS$:
$$
|\sS|=(p^r-1)(g-1),
$$
where $r=[F_{\mathfrak p}:\Q_p]$ and $g$ is the genus of $M_0$.
\end{theorem}
As $M_0$ may not be geometrically connected, the genus is defined to
be one plus the half of the summation of the degree of the canonical
class of each component in $M_0\otimes \bar k$.
\begin{remark}
The generic Newton polygon in $M_0(\bar k)$ is also determined (see
Theorem \ref{classification and existence of newton polygon}). For
the original example of D. Mumford, i.e., $d=3$ and $\Cor_{F|\Q}(D)$
split, R. Noot \cite{Noot1}-\cite{Noot2} (see particularly
Proposition 3.6 \cite{Noot1} and Proposition 2.2 \cite{Noot2})
classified the possible Newton polygons for the mod $p$ reduction of
an abelian variety defined over a number field appeared as a closed
fiber of a Mumford family. Compared with his method, the new point
here is a natural decomposition of the $p$-adic Galois
representation into a tensor product of two dimensional potentially
crystalline $\Q_{p^r}$-representations after tensorizing with
$\Q_{p^r}, r\leq 3$, and this is true for a general Mumford family.
In our approach, the classification result becomes a simple
consequence of the admissibility of a filtered $\phi$-module
associated with a crystalline representation.
\end{remark}

The above result settles then Conjecture 1.3 \cite{SZZ} for Mumford
families. In our previous work loc. cit., we have studied a certain
Shimura curve of PEL type: Deligne-Shimura's mod\`{e}le \'{e}trange
\cite{De1}. However, the old technique does not suffice for the
current situation (see \S1-\S5 loc. cit.). One main reason is that a
Mumford family is characterized by extra Hodge cycles in the generic
fiber which are not yet known to be algebraic in general. It makes impossible a
direct proof of an expected direct-tensor decomposition of the
universal filtered Dieudonn\'{e} module into rank two filtered
crystals over the global base with predicted relative Frobenius
actions on factors. Instead, we have to work with both the category
of \'{e}tale local systems and the category of families of filtered
Frobenius crystals, and use various comparison results in the
$p$-adic Hodge theory. What we have achieved is the following
result.
\begin{theorem}[Theorem \ref{theorem on tensor decomposition over local base}]\label{formal tensor decomposition}
Let $x_0$ be a closed point of $M$ and $\hat{M}_{x_0}$ the
completion of $M $ at $x_0$. Then one has a direct-tensor
decomposition in the category
$\mathcal{MF}_{[0,1]}^{\nabla}(\hat{M}_{x_0})$ of the restriction to
$\hat{M}_{x_0}$ of the universal filtered Dieudonn\'{e} module
attached to $f$:
$$
(H^1_{dR},F_{hod},\nabla^{GM},\phi)|_{\hat{M}_{x_0}}
\cong\{[\otimes_{i=0}^{r-1}(\sN_i,Fil^1_{\sN_i},\nabla_{\sN_i}),\phi_{ten}]\otimes
(M_{A_2},Fil^1_{A_2},d,\phi_{A_2})\}^{\oplus 2^{\epsilon(D)}},
$$
where $\{(\sN_i,Fil^1_{\sN_i},\nabla_{\sN_i})\}_{0\leq i\leq r-1}$
are eigen-components of the universal filtered Dieudonn\'{e} module
of a versal deformation of a Drinfel'd $\sO_{\mathfrak p}$-divisible
module and $\phi_{ten}$ is the tensor product of the $\phi_i$s on
eigen components, and $(M_{A_2},Fil^1_{A_2},d, \phi_{A_2})$ is a
constant unit crystal.
\end{theorem}
Here $\epsilon(D)$ is equal to 0 or 1 which depends on $D$ only (see
\S\ref{Section on Shimura curve of Hodge type}). The above result shows an intimate relation between the associated $p$-divisible
groups to a Mumford family and Drinfel'd $\sO_{\mathfrak
p}$-divisible modules, which we intend to understand
in more depth in the future.\\

The paper is structured as follows. In \S\ref{Section on Shimura
curve of Hodge type} we review briefly the construction of a Shimura
curve of Hodge type arising from the corestriction of a quaternion
division algebra and deduce an integral model of a Mumford family
from the work of Kisin loc. cit.. In \S\ref{Section on two
dimesional crystalline rep} we first show a natural direct-tensor
decomposition of the \'{e}tale $\Z_{p^d}$-local system attached to a
Mumford family into rank two factors, and then show that over each
closed point each factor is potentially crystalline, from which the
classification of the Newton polygons in $M_0(\bar k)$ follows.
Section \ref{Section on formal tensor decomposition} is a bridge
between the classification and the mass formula, in which we study
the universal abelian scheme over the formal neighborhood a
$k$-rational point of the base curve by using the deformation theory
of a $p$-divisible group with Tate cycles due to G. Faltings \S7
\cite{Fa3} (see also \S4 \cite{Mo} and \S 1.5 \cite{Ki}) and prove
the direct-tensor decomposition result Theorem \ref{formal tensor
decomposition}. In the final section we prove the mass formula for
the supersingular locus in $M_0(\bar k)$. For a technical reason, we
consider instead a second tensor power of the universal filtered
Dieudonn\'{e} crystal and construct from a direct factor in the
corresponding decomposition a nonzero morphism $\tilde F_{rel}:
F_{M_0}^{*r}\sP_0\to \sP_0$ in char $p$ whose reduced zero divisor
coincides with the supersingular locus, where $\sP_0$ is a line
bundle of negative degree over $M_0$. With the aid of Theorem
\ref{formal tensor decomposition}, we apply the theory of display to
compute that the multiplicity of the zero divisor is everywhere two
and hence obtain the mass formula. \\

{\bf Acknowledgements:} The work, especially \S\ref{subsection on
tensor decomposition}, benefits from discussions with Jean-Marc
Fontaine during his stay in Mainz in April 2010. We would like to
thank him heartily. We thank Yi Ouyang and Liang Xiao for their help
on Theorem \ref{HT implies de-rham }, and Guitang Lan for his
careful reading of the paper. We would like also to thank Stefan M\"{u}ller-Stach for his interest in this work. Last but not least, we want to thank the referee for pointing out a flaw in the original statement of Lemma \ref{direct summan of dual crystalline sheaf}.

\section{Corestriction of a quaternion division algebra and an integral model of a Mumford family}\label{Section on Shimura curve of Hodge type}
Let $F$ be a totally real field of degree $d\geq 3$ and $D$ a
quaternion division algebra over $F$, which is split only at one
real place of $F$. We denote the set of real embeddings of $F$ by
 $\Psi:=\{\tau=\tau_1,\cdots,\tau_d\}$, and assume that $D$ is split over
$\tau$. Let $\bar \Q$ be the algebraic closure of $\Q$ in $\C$ and
$\Gal_{\Q}$ the absolute Galois group of $\Q$. Recall (see \S4
\cite{Mu}) that the corestriction $\Cor_{F|\Q}D$ is defined as the
subalgebra of $\Gal_{\Q}$-invariant elements of
 $
\bigotimes_{i=1}^{d}D\otimes_{F,\tau_i}\bar \Q
 $.
For it one has the following result:
\begin{lemma}[Lemma 5.7 (a) \cite{VZ}]\label{structure of
corestriction} Let $F$ and $D$ be as above. It holds that either
\begin{itemize}
    \item [(i)] $\Cor_{F|\Q}(D)\cong M_{2^d}(\Q)$ and $d$ is an odd number $\geq 3$,
    or
    \item [(ii)] $\Cor_{F|\Q}(D)\ncong M_{2^d}(\Q)$. Then
    $$
\Cor_{F|\Q}(D)\otimes_{\Q}\Q(\sqrt{b})\cong M_{2^d}(\Q(\sqrt{b})),
    $$
    where $\Q(\sqrt{b})$ is a quadratic
    field extension of $\Q$.
\end{itemize}
\end{lemma}
Both cases can be written uniformly into $
\Cor_{F|\Q}(D)\otimes_{\Q}\Q(\sqrt{b})\cong M_{2^d}(\Q(\sqrt{b}))
 $
for a square free rational number $b\in \Q$, and such an isomorphism
will be fixed in the following. We define a number $\epsilon(D)$
to be 0 in Case (i) and 1 in Case (ii). So we can
fix an embedding $\Cor_{F|\Q}(D)\hookrightarrow
M_{2^{d+\epsilon(D)}}(\Q)$ of $\Q$-algebras. Note that the case
$d=3$ and $\epsilon(D)=0$ is the original example considered by
Mumford loc. cit.. Recall also that one comes along with a natural
morphism of $\Q$-groups
 $$
{\rm Nm}: D^*\to \Cor_{F|\Q}(D)^*, \ d\mapsto (d\otimes
1)\otimes\cdots \otimes(d\otimes 1).
 $$
So one obtains a linear representation ${\rm Nm}: D^*\to
\GL_{2^{d+\epsilon(D)},\Q}$ of the $\Q$-group $D^*$. It gives rise
to a Shimura curve of Hodge type, which is not PEL type (see
Construction 5.8 page 269-273 \cite{VZ} and \S1.1 \cite{Noot2}). Put
$\tilde{G'}_{\Q}:=\{x\in D|\ {\rm Norm}(x)=1\}$ and $\tilde
G_{\Q}:=\G_{m,\Q}\times \tilde{G'}_{\Q}$, and write $\GL_{\Q}$ for
$\GL_{2^{d+\epsilon(D)},\Q}$. The $\Q$-group $G_{\Q}$ is defined to
be image of the morphism $ \tilde G_{\Q}\to \GL_{\Q}$, that is the
product of the natural morphism $\G_{m,\Q}\to \GL_{\Q}$ and ${\rm
Nm}|_{\tilde{G'}_{\Q}}$. It is connected and reductive. The natural
morphism $N: \tilde G_{\Q}\to G_{\Q}$ is a central isogeny. Let
$G'_{\Q}$ be the image of $\tilde G'_{\Q}$ in $G_{\Q}$. The natural
embedding $G_{\Q}\hookrightarrow \GL_{\Q}$ factors through
$\GSp_{\Q}\subset \GL_{\Q}$, which can be seen as follows: let
$H_{\Q}:=\Q(\sqrt{b})^{2^d}$ be a $\Q$-vector space with the
$\Cor_{F|\Q}(D)$ action by the left multiplication and $\G_{m,\Q}$
action by scalar multiplication. This induces a $G_{\Q}$-action on
$H_{\Q}$. It is easy to verify that there exists a
$\Q(\sqrt{b})$-valued symplectic form $\omega$ on $H_{\Q}$, unique
up to scalar, which is invariant under the $G'_{\Q}$-action.
Then $G_{\Q}=\G_{m,\Q}\cdot \tilde G'_{\Q}\subset \GL_{\Q}$ acts
on the $\Q$-valued symplectic form $
\psi:=\tr_{\Q(\sqrt{b})|\Q}\omega$ by similitude. Let $S^1$ be the
real group $\{z\in \C|\  z\bar z=1\}$. One defines
$$
u_0: S^1\to \tilde{G'}_{\R}(\R)\cong\SL_{2}(\R)\times \SU(2)^{\times d-1},\ e^{i\theta}\mapsto \left(%
\begin{array}{cc}
  \cos \theta & \sin \theta \\
  -\sin \theta & \cos \theta \\
\end{array}%
\right)\times id^{\times d-1}.
$$
The morphism $\tilde{h}_0=id\times u_0: \R^*\times S^1\to
\tilde{G}_{\R}$ descends to a morphism of real groups:
$$
h_0: \SSS=\Res_{\C|\R}\G_m\to G_{\R}.
$$
Let $X$ be the $G(\R)$-conjugacy class of $h_0$ and
$(\GSp(H_{\Q},\psi),X(\psi))$ the Siegel space defined by
$(H_{\Q},\psi)$. One verifies that
 $(G_{\Q},X)\hookrightarrow(\GSp(H_{\Q},\psi),X(\psi))$  is a morphism of Shimura datum, and therefore defines a
Shimura curve of Hodge type. Now let $K\subset G(\A_f)$ be a compact
open subgroup and one defines the Shimura curve as the double coset
$$
Sh_{K}(G,X):= G(\Q)\backslash X\times G(\A_f)/ K,
$$
where
$$
q(x,a)b=(qx,qab), \ q\in G(\Q), x\in X, a\in G(\A_f), b\in K.
$$
By the theory of canonical model, $M_K:=Sh_{K}(G,X)$ is naturally
defined over the reflex field of $(G,X)$, that is $\tau(F)\subset
\C$ in this current case. It is not difficult to show that $M_K$ is
proper over $F$.\\

Now let $p$ be a rational prime satisfying Assumption
\ref{assumption on p from integral model} in the introductory part
and $p\sO=\prod_{i=1}^{n}\mathfrak{p}_i$ the prime decomposition of
$p$ in $F$. By choosing an embedding $\iota: \bar \Q\hookrightarrow
\bar \Q_p$ (which is fixed once and for all), one gets an
identification of $\Psi$ with $$\Hom_{\Q}(F,\bar \Q_p)=
\coprod_{i=1}^{n}\Hom_{\Q_p}(F_{\mathfrak{p}_i},\bar \Q_p).$$ Write
$F_i$ for $F_{\mathfrak{p}_i}$ and put $r_i:=[F_i:\Q_p]$. We assume
that $\tau\in \Hom_{\Q_p}(F_{1},\bar \Q_p)$. Set
$\mathfrak{p}=\mathfrak{p}_1$ and $r=r_1$. The condition on $p$
implies that $G_{\Q_p}$ is quasi-split and split over an unramified
extension of $\Q_p$. Hence hyperspecial subgroups exist in $G(\Q_p)$
(see 1.10 \cite{Ti}). Recall that we have a central isogeny $\tilde
G_{\Q}\to G_{\Q}\subset \GL_{\Q}$ over $\Q$ with $\tilde
G_{\Q}=\G_{m,\Q}\times \tilde G'_{\Q}$, where $\tilde G'_{\Q}=\ker\
({\rm Norm:}D^*\to F^*)$. The assumption on $p$ implies that
$D^*(\Q_p)\cong \prod_{i=1}^{n}\GL_2(F_{i})$. It is
clear that ${\rm Norm}\otimes_{\Q}\Q_p$ becomes a product of the
determinants under the isomorphism. So this implies that $\tilde
G'(\Q_p)\cong\prod_{i=1}^{n}\SL_2(F_{i})$, and hence
$\tilde G(\Q_p)\cong\Q_p^*\times
\prod_{i=1}^{n}\SL_2(F_{i})$. Thus a hyperspecial
subgroup of $G(\Q_p)$ is conjugate to the image of $\Z_p^*\times
\prod_{i=1}^{n}\SL_2(\sO_{F_{i}})\subset \tilde
G(\Q_p)$ under the isogeny $\tilde G(\Q_p)\to G(\Q_p)$. In what
follows the $p$-component $K_p\subset G(\Q_p)$ of the level
structure $K$($=K_pK^p\subset G(\Q_p)G(\A_f^p)$) is always taken to
be hyperspecial. The main result of Kisin \cite{Ki} asserts then
that, for our chosen prime $\mathfrak p|p$, \emph{the} integral
canonical model $\sM_K$ of $M_K$ exists, which is a smooth
$\sO_{(\mathfrak p)}$-scheme for $K^p$ sufficiently small. The
construction of $\sM_K$ (see \S2.3 loc. cit.) provides with a
universal abelian scheme over $\sM_K$ as well, once the coprime to
$p$-component $K^p$ is chosen small enough: take a suitable maximal
order $\sO_D$ of the $F$-algebra $D$ and consider
$$
\Cor_{F|\Q}\sO_D:=(\otimes_{i=1}^{d}\sO_D\otimes_{\sO_F,\tau_i}\bar{\Z})^{\Gal_{\Q}}\subset
\Cor_{F|\Q}D.
$$
There exists a lattice $H_{\Z}\subset H_{\Q}$, which is
stabilized by $\Cor_{F|\Q}\sO_D\otimes_{\Z}\sO_{\Q(\sqrt b)}$, such
that there is a closed embedding $G_{\Z_p}\hookrightarrow
\GL(H_{\Z_p})$ (where $G_{\Z_p}$ is the reductive group scheme over
$\Z_p$ associated with $K_p$) whose generic fiber is the base change
to $\Q_p$ of $G_{\Q}\hookrightarrow \GL_{\Q}$. Let $K'_p\subset
\GSp(\Q_p)$ be the stabilizer of $H_{\Z_p}$. One can choose a
$K'^p\subset \GSp(\A_f^p)$ such that for $K'=K'_pK'^p$ one has an
embedding of Shimura varieties
 $
M_K\hookrightarrow Sh_{K'}(\GSp(\psi),X(\psi))
 $ and $Sh_{K'}(\GSp(\psi),X(\psi))$ has
an integral model  $\sS_{K'}:=\sS_{K'}(\GSp(\psi),X(\psi)) $ over
$\Z_{(p)}$ (which is not necessarily smooth) representing a moduli
functor over $\Z_{(p)}$ (see \S2.3.3 loc. cit.). As $K^p$ is
required to be small enough, one may further assume that $K'^p$ is
taken so small that there exists an abelian scheme $\sA_{K'}\to
\sS_{K'}$ over $\sS_{K'}$. Recall (see Theorem 2.3.8 loc. cit.) that
$\sM_{K}$ is defined as the normalization of the closure of the
composite
$$
M_K\hookrightarrow Sh_{K'}(\GSp(\psi),X(\psi))\hookrightarrow
\sS_{K'}\times_{\Z_{(p)}}\sO_{(\mathfrak p)}.
$$
Now we define our abelian scheme $f_K: \sX_K\to \sM_K$ to be the
morphism sitting in the Cartesian diagram:
$$
\xymatrix{
          \sX_K\ar[d]_{f_K} \ar[r]_{}     &   \sA_{K'}\times_{\Z_{(p)}}\sO_{(\mathfrak p)}      \ar[d]^{ }     \\
    \sM_K\ar[r]_{ } & \sS_{K'}\times_{\Z_{(p)}}\sO_{(\mathfrak p)}.             }
$$
For the sake of convenience, we shall change our foregoing notation
as follows: let $M$ (resp. $f: X\to M$) be the completion of the
integral canonical model $\sM_K$ (resp. $f_K: \sX_K\to \sM_K$) at
$\mathfrak p$. They are defined over the discrete valuation ring
$\sO_{\mathfrak p}$, the completion of $\sO_{(\mathfrak p)}$ at the
maximal ideal. The superscript (resp. subscript) zero on an object
means the base change of the object to the generic (resp. closed)
fiber of $\sO_{\mathfrak p}$.\\
To the universal abelian scheme $f:X\to M$ we attach the \'{e}tale
$\Z_p$-local system $\HH:=R^1f^0_{*}(\Z_{p})_{\bar{X}_{et}^0}$ over
$M^0$ and the universal filtered Dieudonn\'{e} module $(H,F,
\nabla,\phi):=(H^1_{dR},F_{hod},\nabla^{GM},\phi)$, which is an
object in the category $\mathcal{MF}^{\nabla}_{[0,1]}(M)$ introduced
by Faltings (see Ch. II \cite{Fa2} and \S3 \cite{Fa3}). To
distinguish the notation, the $p$ torsion analogue of the previous
category will be denoted by
$\mathcal{MF}^{\nabla}_{[0,a]}(M)_{tor}$. In loc. cit. Faltings
constructed a fully faithful functor $\DD$ from
$\mathcal{MF}^{\nabla}_{[0,p-2]}(M)$ (resp.
$\mathcal{MF}^{\nabla}_{[0,p-2]}(M)_{tor}$) to the category of
\'{e}tale $\Z_p$ (resp. $p$-torsion) local systems over $M^0$. By
the Remark after Theorem 2.6* in \cite{Fa2}, one has
$\DD(\sO_{X}/p^n,d)=\Z/p^n$ for each $n\in \N$. Applying Theorem 6.2
and Remark after the theorem loc. cit. on the compatibility of the
direct image with the functor $\DD$, one gets
 $
\DD(H/p^n,F, \nabla,\phi)=\HH^\vee/p^n
 $.
By taking the inverse limit, one obtains then $\DD(H,F,
\nabla,\phi)=\HH^\vee$. One notices that the information on the Newton
jumping locus of $f_0: X_0\to M_0$ is encoded in the attached
universal filtered Dieudonn\'{e} module while the defining
information of a Mumford family is basically contained in the
\'{e}tale local system over $M^0$.

\section{Two dimensional potentially crystalline $\Q_{p^r}$-representations and classification of the Newton
polygons}\label{Section on two dimesional crystalline rep} For a
$k$-rational point $x_0$ of $M_0$, the closed fiber of $f_0$ at
$x_0$ is denoted by $A_{x_0}$. The Newton polygon of $A_{x_0}$ is
defined to be the Newton polygon of its associated $p$-divisible
group. The aim of this section is to determine the possible Newton
polygons of $A_{x_0}$ when $x_0$ varies in $M_0(\bar k)$.
\subsection{Tensor decomposition of the Galois
representation}\label{subsection on tensor decomposition} Let $x_0$
be as above. Because $M$ is smooth over $\sO_{\mathfrak p}$, there
exists an $\sO_{W(k)}$-valued point $x$ of $M$ which lifts $x_0$.
Let $A_{x}$ be the corresponding abelian scheme over $\sO_{W(k)}$
whose reduction is equal to $A_{x_0}$. The aim of this paragraph is
to show a certain direct-tensor decomposition of the $p$-adic Galois
representation associated with the generic fiber $A_{x^0}$ of $A_x$.
\begin{lemma}\label{tensor decomposition of corestriction algebra over local
fields} One has a natural isomorphism of $\Q_p$-algebras:
$$
\Cor_{F|\Q}(D)\otimes_{\Q}\Q_p\cong\otimes_{i=1}^{n}\Cor_{F_i|\Q_p}(D\otimes_FF_i).
$$
\end{lemma}
\begin{proof}
Put $D_{i}=D\otimes_{F,\tau_i}\bar \Q, 1\leq i\leq d$. For an
element $a\otimes \lambda \in D_i$ and $g\in \Gal_{\Q}$,
$$
g(a\otimes_{\tau_i}\lambda)=a\otimes_{g(\tau_i)}g(\lambda).
$$
By the definition of the corestriction, one has a natural
isomorphism of $\Gal_{\Q}$-modules:
$$
\Cor_{F|\Q}(D)\otimes_{\Q}\bar{\Q}\cong\otimes_{i=1}^{d}D_i.
$$
Let $D_{\iota}$ be the decomposition group of $\iota$ in
$\Gal_{\Q}$, which is isomorphic to the local Galois group
$\Gal_{\Q_p}$. Now we consider the $D_{\iota}$-invariants of two
sides of the above isomorphism after tensorizing with $\bar{\Q}_p$
via $\iota$. Obviously one obtains $\Cor_{F|\Q}(D)\otimes_{\Q}\Q_p$
from the left side. Let
$$
O_1:=\{\tau_1=\tau,\cdots,\tau_{r_1}=\tau_{r}\},\cdots,
O_n=\{\tau_{r_1+\cdots+r_{d-1}+1},\cdots,\tau_{r_1+\cdots+r_{d-1}+r_d}=\tau_n\}
$$
be the $n$-orbits of $D_{\iota}$-action on $\Psi$. Note that there
is a natural isomorphism
$$
(\otimes_{\tau_j\in
O_i}D_j\otimes_{\bar{\Q},\iota}\bar{\Q}_p)^{D_{\iota}}\cong\Cor_{F_i|\Q_p}(D\otimes_FF_i).
$$
As the tensor product $\otimes_{i=1}^{n}(\otimes_{\tau_j\in
O_i}D_j\otimes_{\bar{\Q},\iota}\bar{\Q}_p)^{D_{\iota}}$ over
$\Q_p$ is clearly a subspace of the $D_{\iota}$-invariants on the
right side, it must be the whole invariant space for the dimension
reason. So the lemma follows.
\end{proof}
Consider the base change to $\Q_p$ of the $\Q$-morphism ${\rm Nm}:
D^*\to \Cor_{F|\Q}(D)^*$. The following statement is clear from
the proof of the last lemma.
\begin{lemma}
The morphism ${\rm Nm}_{\Q_p}: D^*(\Q_p)\to
\Cor_{F|\Q}(D)^*(\Q_p)$ factors through the natural morphism
$$
\prod_{i=1}^{n}\Cor_{F_i|\Q_p}(D\otimes_FF_i)^*\to
\Cor_{F|\Q}(D)^*(\Q_p).
$$
Moreover, under the natural decomposition
$D^*(\Q_p)=\prod_{i=1}^{n}(D\otimes_FF_{i})^*$, ${\rm Nm}_{\Q_p}$
is written as a product $\prod_{i=1}^{n}{\rm Nm}_i$ where for each
$i$ the morphism
$$
{\rm Nm}_i: (D\otimes_FF_i)^*\to \Cor_{F_i|\Q_p}(D\otimes_FF_i)^*
$$
is the natural diagonal morphism for the corestriction.
\end{lemma}
As a consequence, the representation of $D^*(\Q_p)$ on $H_{\Q_p}$
admits a natural tensor decomposition: by Schur's lemma the
representation decomposes as a tensor product. In the current
situation, this can be seen in a direct way: by Assumption
\ref{assumption on p from integral model}, $D\otimes_FF_i$ splits
for each $i$, and so does $\Cor_{F_i|\Q_p}(D\otimes_FF_i)$, which
is isomorphic to $M_{2^{r_i}}(\Q_p)$. In the case
$\epsilon(D)=0$, the morphism
$$
\prod_{i=1}^{n}\Cor_{F_i|\Q_p}(D\otimes_FF_i)^*\to
\Cor_{F|\Q}(D)^*(\Q_p)=\GL(H_{\Q_p})
$$
is isomorphic to the tensor product morphism
$$
\prod_{i=1}^{n}\GL_{2^{r_i}}(\Q_p)\longrightarrow \GL_{2^d}(\Q_p),
\ (g_1,\cdots,g_n)\mapsto g_1\otimes \cdots\otimes g_n.
$$
In the case $\epsilon(D)=1$, $\Cor_{F|\Q}(D)$ is nonsplit and it
splits after tensorizing with $\Q(\sqrt b)$. Consider the composite
\begin{eqnarray*}
  \prod_{i=1}^{n}\Cor_{F_i|\Q_p}(D\otimes_FF_i)^* &\to&\Cor_{F|\Q}(D)^*(\Q_p)\subset(\Cor_{F|\Q}(D)\otimes_{\Q}\Q(\sqrt
b))^*(\Q_p)  \\
  &=&\GL_{\Q(\sqrt b)}(H_{\Q})(\Q_p)\subset \GL(H_{\Q})(\Q_p).
\end{eqnarray*}
The above inclusion $\Cor_{F|\Q}(D)^*\subset
(\Cor_{F|\Q}(D)\otimes_{\Q}\Q(\sqrt b))^*$ is given by $a\mapsto
a\otimes 1$. One has the following commutative diagram
$$
\xymatrix{
  \Cor_{F|\Q}(D)^*(\Q_p)\times \Q(\sqrt b)^*(\Q_p)\ar[d]_{\cap} \ar[r]^{\quad\quad\quad}
                & \GL_{\Q(\sqrt b)}(H_{\Q})(\Q_p) \ar[d]^{\cap}  \\
  \Cor_{F|\Q}(D)^*(\Q_p)\times
\GL_2(\Q_p)  \ar[r]_{\quad\quad\quad}
                &    \GL(H_{\Q})(\Q_p).          }
$$
The image of $\prod_{i=1}^{n}\Cor_{F_i|\Q_p}(D\otimes_FF_i)^*$ in
the left-up element of the above diagram is contained in
$\Cor_{F|\Q}(D)^*(\Q_p)\times \{1\}$. Thus the morphism $$
\prod_{i=1}^{n}\Cor_{F_i|\Q_p}(D\otimes_FF_i)^*\to\GL(H_{\Q_p})$$
is isomorphic to the composite of the obvious morphisms
$$
\prod_{i=1}^{n}\GL_{2^{r_i}}(\Q_p)\hookrightarrow
\prod_{i=1}^{n}\GL_{2^{r_i}}(\Q_p)\times
\GL_2(\Q_p)\stackrel{\otimes}{\longrightarrow}
\GL_{2^{d+1}}(\Q_p).
$$
For $s\in \N$, one denotes by $\sigma\in \Gal(\Q_{p^s}|\Q_p)$ the
Frobenius automorphism. For a topological group $P$ with a
continuous $\Z_{p^s}$ linear representation $W$, the
$\sigma^i$-conjugate $W_{\sigma^i}$ of $W$ for $0\leq i\leq s-1$
is defined to be the tensor product
$W\otimes_{\Z_{p^s},\sigma^i}\Z_{p^s}$, where $P$ acts on
$\Z_{p^s}$ trivially. Similarly define the
$\sigma^{\cdot}$-conjugations of a $\Q_{p^s}$-representation. The
symbol $\otimes_{\sigma^i}$ signifies the equalities of two
tensors:
$$
\lambda(x\otimes \mu)=x\otimes \lambda \mu,\ \lambda x\otimes\mu =
x\otimes \lambda^{\sigma^i}\mu , \ {\rm for}\ \lambda,\mu\in
\Z_{p^s}, x\in W.
$$
Consider the morphism
$$
{\rm Nm}_1: (D\otimes_FF_1)^*\longrightarrow
\Cor_{F_1|\Q_p}(D\otimes_FF_1)^*, \ a\mapsto
(a\otimes_{F_1,\tau_1}1)\otimes\cdots\otimes(a\otimes_{F_1,\tau_r}1).
$$
Note that for $1\leq i\leq r$, $\tau_i(F_1)=\Q_{p^r}$, the unique
unramified extension of $\Q_p$ of degree $r$ in $\bar{\Q}_p$. Then
one has a natural isomorphism
$$
\Cor_{F_1|\Q_p}(D\otimes_FF_1)\otimes_{\Q_p}F_1\cong\otimes_{i=0}^{r-1}(D\otimes_FF_1\otimes_{F_1,\sigma^
i}F_1).
$$
This implies that the natural morphism
 $(D\otimes_FF_1)^*\longrightarrow \Cor_{F_1|\Q_p}(D\otimes_FF_1)^*(F_1)$ is
isomorphic to the composite of
$$
\GL_2(\Q_{p^r})\hookrightarrow \prod_{i=0}^{r-1} \GL_2(\Q_{p^r}),
\ g\mapsto (g,\cdots,\sigma^i(g),\cdots,\sigma^{r-1}(g))
$$
with the tensor product morphism $ \prod_{i=0}^{r-1}
\GL_2(\Q_{p^r})\longrightarrow \GL_{2^r}(\Q_{p^r})
 $.
Summarizing the above discussions, we derive the following
\begin{lemma}\label{tensor decomposition wrt D star}
The representation of $D^*(\Q_p)$ on $H_{\Q_p}$ admits a natural
tensor decomposition
$$
H_{\Q_p}=(V_{\Q_p}\otimes U_{1,\Q_p}\otimes\cdots\otimes
U_{n-1,\Q_p})^{\oplus 2^{\epsilon(D)}}.
$$
Moreover, the representation $V_{\Q_p}\otimes_{\Q_p}\Q_{p^r}$
decomposes further into a tensor product $
V_{1}\otimes_{\Q_{p^r}}V_{1,\sigma}\otimes_{\Q_{p^r}}
\cdots\otimes_{\Q_{p^r}} V_{1,\sigma^{r-1}}
 $
with $\dim_{\Q_{p^r}}V_1=2$. For $1\leq i\leq n-1$,
$U_{i,\Q_p}\otimes_{\Q_p}\Q_{p^{r_{i+1}}}$ decomposes into a
tensor product in a similar manner.
\end{lemma}

\begin{lemma}\label{lemma on lattice decomposition}
Let $P$ be a topological group together with a continuous linear
representation on a finite dimensional $\Q_{p^s}$-vector space
$W$. Assume the following two conditions hold:
\begin{itemize}
    \item [(i)] The representation factors as
$$
P\to \GL(W_1)\times \SL(W_2)\stackrel{\otimes}{\longrightarrow}
\GL(W),
$$
where $W_i, i=1,2$ are two $\Q_{p^s}$-vector spaces.
    \item [(ii)] There is a $\Z_{p^s}$-lattice $W_{\Z_{p^s}}$ in
$W$, which is stable under the $P$-action and
    admits a lattice tensor decomposition
    $$W_{\Z_{p^s}}=W_{1,\Z_{p^s}}\otimes_{\Z_{p^s}}W_{2,\Z_{p^s}},$$
    where $W_{i,,\Z_{p^s}}$ is a $\Z_{p^s}$-lattice of $W_i$ for $i=1,2$.
\end{itemize}
Then in the factorization of (i), the lattice $W_{i,,\Z_{p^s}}$
for $i=1,2$ is stable under the $P$-action on $W_i$.
\end{lemma}
\begin{proof}
Consider the following commutative diagram:
$$
\xymatrix{&\GL(W_{1,\Z_{p^s}})\times
\SL(W_{2,\Z_{p^s}})\ar[r]^{\quad \quad\quad\otimes }\ar[d]_{\cap}
&\GL(W_{\Z_{p^s}})\ar[d]^{\cap}\\
 P\ar[r]^{}\ar[r]^{}&\GL(W_1)\times \SL(W_2)
\ar[r]^{\quad\quad\otimes}&\GL(W).}
$$
It suffices to show that the representation $P\to
\GL(W_{\Z_{p^s}})\subset \GL(W)$ factors through
 $$\GL(W_{1,\Z_{p^s}})\times \SL(W_{2,\Z_{p^s}})\to
\GL(W_{\Z_{p^s}}).$$ Note that $\GL(W_{\Z_{p^s}})$ is a compact
subgroup of $\GL(W)$. As the morphism $\otimes$ has a finite
kernel,
 $$
T:=\otimes^{-1}(\GL(W_{\Z_{p^s}})\cap \otimes(\GL(W_1)\times
\SL(W_2)))
 $$
is a compact subgroup of $\GL(W_1)\times \SL(W_2)$. Since $T$
contains $\GL(W_{1,\Z_{p^s}})\times \SL(W_{2,\Z_{p^s}})$, which is
maximal compact, it holds that $T=\GL(W_{1,\Z_{p^s}})\times
\SL(W_{2,\Z_{p^s}})$. Since the image of $P$ in $\GL(W_1)\times
\SL(W_2)$ is contained in $T$ by assumption, the morphism $P\to
\GL(W_{\Z_{p^s}})$ factors through
 $\GL(W_{1,\Z_{p^s}})\times \SL(W_{2,\Z_{p^s}})\to
\GL(W_{\Z_{p^s}}) $.  This proves the lemma.
\end{proof}
\begin{proposition}\label{tensor decomposition of linear
representation} The $K_p$-representation $H_{\Z_p}$ admits a
natural direct-tensor decomposition
$$H_{\Z_p}=(V\otimes U)^{\oplus 2^{\epsilon(D)}},$$ where $U$
decomposes into $U=U_1\otimes\cdots \otimes U_{n-1}$. The tensor
factor $V$ after tensorizing with $\Z_{p^r}$ decomposes further into
$$
V\otimes_{\Z_{p}} \Z_{p^r}=
V_{1}\otimes_{\Z_{p^r}}V_{1,\sigma}\otimes_{\Z_{p^r}}
\cdots\otimes_{\Z_{p^r}} V_{1,\sigma^{r-1}}.
$$
Similar for other tensor factors $U_i, 1\leq i\leq n-1$ after
tensorizing with $\Z_{p^{r_{i+1}}}$.
\end{proposition}
\begin{proof}
Recall that $K_p$ is conjugate to the image of
$$
\tilde K_p:=(\Z_p^*\times \SL_2(\sO_{F_{\mathfrak{p}}}))\times
\prod_{i=2}^{n}\SL_2(\sO_{F_{i}})\subset \tilde
G(\Q_p)
$$
under the map $N_{\Q_p}: \tilde G(\Q_p)\to G(\Q_p)$. The
direct-tensor decomposition of $H_{\Q_p}$ as
$D^*(\Q_p)$-representation in Lemma \ref{tensor decomposition wrt
D star} induces a direct-tensor decomposition of $H_{\Q_p}$ as
$\tilde K_p$-representation. Since the $\tilde K_p$-action on
$H_{\Q_p}$ factors through the $K_p$-action on $H_{\Q_p}$ by
definition, one obtains the direct-tensor decomposition of
$H_{\Q_p}$ for the $K_p$-action as well. By the definition of the
lattice $H_{\Z}$, it is easy to see that $H_{\Z_p}$ decomposes
into a direct-tensor product of $\Z_p$-lattices. Then it is also a
direct-tensor decomposition as $K_p$-representation by Lemma
\ref{lemma on lattice decomposition}. The proofs of the tensor
decompositions for the factors $V$ and $U_{\cdot}$ are analogous.
\end{proof}
By Proposition 2.2.4 \cite{Ki}, each geometrically connected
component of $M^0$ is defined over an unramified extension of
$F_{\mathfrak p}$. Since there is a finite number of them, we can
fix a finite extension $L$ of $F_{\mathfrak p}$ inside the maximal
unramified extension $\Q_{p}^{ur}$ such that each component
defines and admits an $L$-rational point.
\begin{corollary}\label{tensor decompositon of local system}
One has a direct-tensor decomposition of \'{e}tale local systems
over $M^0\times_{F_{\mathfrak p}} L$:
$$
\HH=(\V\otimes \U)^{\oplus 2^{\epsilon(D)}}, \U=\U_1\otimes\cdots
\otimes \U_{n-1}
$$
and
$$
\V\otimes_{\Z_{p}}
\Z_{p^r}\cong\V_{1}\otimes_{\Z_{p^r}}\V_{1,\sigma}\otimes_{\Z_{p^r}}
\cdots\otimes_{\Z_{p^r}} \V_{1,\sigma^{r-1}},
$$
where for $0\leq i\leq r-1$, $\V_{1,\sigma^{i}}$ is the
$\sigma^i$-conjugate of $\V_1$. Similarly for $\U_i, 1\leq i\leq
n-1$ one has
$$
\U_i\otimes_{\Z_{p}}
\Z_{p^{r_{i+1}}}\cong\U_{i,1}\otimes_{\Z_{p^{r_{i+1}}}}\U_{i,\sigma}\otimes_{\Z_{p^{r_{i+1}}}}
\cdots\otimes_{\Z_{p^{r_{i+1}}}} \U_{1,\sigma^{r_{i+1}-1}}.
$$
\end{corollary}
\begin{proof}
Let $M^0\times_{F_{\mathfrak p}} L=\bigsqcup_i M^0_i$ be the
disjoint union of its geometrically connected components, and let
$M^0_1$ be the component which is represented by the double coset
$[1]\in G(\Q)_+\backslash G(\A_f)/K$. It suffices to show that the
tensor decomposition of the restriction $\HH$ to $M^0_1$. Consider
the short exact sequence of \'{e}tale fundamental groups:
$$
1\to \pi_1^{geo}(M^0_1)\to \pi_1^{arith}(M^0_1)\to \Gal(\bar
\Q_p|L)\to 1.
$$
An $L$-rational point of $M^0$ induces a splitting of the exact
sequence, and one writes
$$
\pi_1^{arith}(M^0_1)=\pi_1^{geo}(M^0_1)\cdot \Gal(\bar\Q_p|L).
$$
To show the tensor decomposition of $\HH|_{M^0_1}$, it is to show
that the factorization of $\pi_1^{geo}(M^0_1)\to K_p$ and
$\Gal(\bar \Q_p|L)\to K_p$. The latter follows from the proof of
Lemma 2.2.1 \cite{Ki}. The former goes as follows: it is known
that $\pi_1^{top}(M^0_1(\C))$ is equal to $K\cap G(\Q)_+$, and the
representation $\pi_1^{geo}(M^0_1)\to \GL(H_{\Z_p})$ is the
composite of
$$
\pi_1^{geo}(M^0_1)\cong\hat{\pi}_1^{top}(M^0_1)\stackrel{\widehat{}}{\to}\GL(H_{\hat{\Z}})\twoheadrightarrow
\GL(H_{\Z_p}).
$$
Obviously the representation $\pi_1^{top}(M^0_1)=K\cap G(\Q)_+\to
\GL(H_{\Z})$ factors through $K\subset \GL(H_{\Z})$. Hence the
result follows from Proposition \ref{tensor decomposition of
linear representation}.
\end{proof}
Specializing the above tensor decompositions of \'{e}tale local
systems into a closed point, one obtains the following
\begin{corollary}\label{tensor decomposition of galois
representation} Let $E$ be a finite extension of $L$ and $x^0$ an
$E$-rational point of $M^0$. Let $H_{\Z_p}=H^1_{et}(\bar A_{x^0},
\Z_p)$ and
 $
\rho: \Gal_E\to \GL(H_{\Z_p})
 $
be the associated Galois representation. Then one has a
direct-tensor decomposition of $\Gal_E$-modules
$$
H_{\Z_p}=(V_{\Z_p} \otimes U_{\Z_p})^{\oplus 2^{\epsilon(D)}},
$$
and a further tensor decomposition of $V_{\Z_p}$ after tensorizing
with $\Z_{p^r}$
$$
V_{\Z_p}\otimes_{\Z_{p}} \Z_{p^r}=
V_{1}\otimes_{\Z_{p^r}}V_{1,\sigma}\otimes_{\Z_{p^r}}
\cdots\otimes_{\Z_{p^r}} V_{1,\sigma^{r-1}}.
$$
\end{corollary}

\subsection{Each tensor factor is potentially crystalline}
It is standard that $H_{\Q_p}$ is a polarisable crystalline
representation of Hodge-Tate weights $\{0,1\}$. In the following
we will show that each factor appearing in the direct-tensor
decomposition of Corollary \ref{tensor decomposition of galois
representation} is potentially crystalline.
\begin{proposition}\label{first crystalline tensor decomposition}
$U_{\Q_p}$ is a potentially unramified representation. As a
consequence, both $V_{\Q_p}$ and $U_{\Q_p}$ are potentially
crystalline.
\end{proposition}
\begin{proof}
Let $I_E\subset \Gal_E$ be the inertia group. We claim that the
image of $I_E$ in $\GL(U_{\Q_p})$ is finite. Assuming the claim,
one sees that $U_{\Q_p}$ is potentially unramified and hence
potentially crystalline. Clearly $V_{\Q_p}\otimes U_{\Q_p}$, as a
direct factor of $H_{\Q_p}$, is crystalline. Therefore $V_{\Q_p}$,
that is a subrepresentation of $V_{\Q_p}\otimes U_{\Q_p}\otimes
U_{\Q_p}^\vee$, is also potentially crystalline. To show the claim,
we introduce the Hodge-Tate cocharacter
 $
\mu_{HT}: \G_m(\C_p)\to G(\C_p)\subset \GL(H_{\C_p})
 $ induced by the Hodge-Tate decomposition of
$H_{\Q_p}$ and the Hodge-de Rham cocharacter
 $
\mu_{HdR}: \G_m(\C)\to G(\C)\subset \GL(H_\C)
 $
induced by the Hodge decomposition of $H^1_{B}(A(\C),\Q)$. Let
$C_{HdR}$ (resp. $C_{HT}$) be the $G(\C)$ (resp.
$G(\C_p)$)-conjugacy class of $\mu_{HdR}$ (resp. $\mu_{HT}$). Then
$C_{HdR}$ is defined over the reflex field $\tau(F)\subset \C$ of
$(G,X)$. It follows from a result of Blasius and Wintenberger (see
Theorem 0.3 \cite{Bl} and Proposition 7 \cite{Wi}, see also
Theorem 4.2 \cite{Og}) that
$$
C_{HT}=C_{HdR}\otimes_{F,\tau}\C_p,
$$
where $\tau: F\to \C_p$ is the composite
 $
F\stackrel{\tau}{\to} \bar \Q\hookrightarrow \bar \Q_p\subset
\hat{\bar{\Q}}_p=\C_p
 $.
Since $\tilde G\to G$ is a central isogeny, there is a natural
number $a$ such that the $a$-th power ${\mu_{HdR}}^a$ (resp.
${\mu_{HT}}^a$) lifts to a cocharacter into $\tilde G(\C)$ (resp.
$\tilde G(\C_p)$). Consider the projection of ${\mu_{HdR}}^a$ to
an $\SL_2$-factor in the decomposition
 $
\tilde G(\C)=\C^*\times \SL_2(\C)\times \cdots\times \SL_2(\C)
 $,
where the order of $\SL_2$-factors is arranged according to
$\Psi$. By the definition of $\tilde h_0$ in \S\ref{Section on
Shimura curve of Hodge type}, one sees that \emph{only} the
projection to the first $\SL_2$-factor (corresponding to $\tau$)
is nontrivial. By the above identification, the same situation
holds for the projections of ${\mu_{HT}}^a$ to $\SL_2$-factors in
the decomposition
$$
\tilde G(\C_p)=\C_p^*\times \SL_2(\C_p)\times \cdots\times
\SL_2(\C_p),
$$
where the order of $\SL_2$-factors is arranged according to
$\Hom_{\Q}(F,\bar{\Q}_p)$ which has been identified with $\Psi$.
By the construction of $U$-factor, the projection of
${\mu_{HT}}^a$ to the factor $\GL(U_{\C_p})$ is trivial. So the
projection of $\mu_{HT}$ to $\GL(U_{\C_p})$ is finite. By S. Sen's
theorem (see \cite{Sen}), the Zariski closure of $\rho(I_E)\subset
G(\Q_p)$ is equal to the $\Q_p$-Zariski closure of $\mu_{HT}$. So
the image of $I_E$ in $\GL(U_{\Q_p})$ is finite.
\end{proof}

For a finite extension $E$ of $\Q_p$ let $E_0\subset E$ be the
maximal unramified subextension. Recall that
\begin{definition}
A $\Q_{p^r}$-representation of $\Gal_E$ is a finite dimensional
$\Q_{p^r}$-vector space $V$ equipped with a continuous action $
\Gal_E\times V\to V
 $
satisfying
$$
g(v_1+v_2)=g(v_1)+g(v_2),\ g(\lambda v)=g(\lambda)g(v)
$$
for $g\in \Gal_E$, $\lambda\in \Q_{p^r}$ and $v,v_1,v_2\in V$. It
is called \emph{Hodge-Tate} (resp. \emph{de-Rham},
\emph{crystalline}) $\Q_{p^r}$-representation if it is so as a
$\Q_p$-representation.
\end{definition}

The following result is known among experts. A variant of it was
communicated by L. Berger to the first named author during the
$p$-adic Hodge theory workshop in ICTP, 2009. The first official
proof should appear in the Ph.D thesis of G. Di Matteo (see the
recent preprint \cite{Ma}). Another proof has been communicated to
us by L. Xiao (see \cite{Xiao}).

\begin{theorem}\label{HT implies de-rham }
Let $V$ and $W$ be two $\Q_{p^r}$-representations of $\Gal_E$. If
$V\otimes_{\Q_{p^r}}W$ is de Rham and one of the tensor factors is
Hodge-Tate, then each tensor factor is de Rham.
\end{theorem}

Applying Theorem \ref{HT implies de-rham } to the tensor factor
$V_{\Q_p}$ in Proposition \ref{first crystalline tensor
decomposition}, one obtains the following
\begin{proposition}\label{second crystalline tensor decomposition}
Making an additional finite field extension $E'\subset E''$ if
necessary, one has a further decomposition of
$\Gal_{E''}$-representations:
$$
V_{\Q_p}\otimes\Q_{p^r}\cong V_1\otimes_{\Q_{p^r}}
\cdots\otimes_{\Q_{p^r}} V_{1,\sigma^{r-1}},
$$
where $\Gal_{E''}$ acts on $\Q_{p^r}$ trivially and
$V_{1,\sigma^i}$ is the $\sigma^{i}$-conjugate of $V_1$. Then
$V_{1}$ is potentially crystalline.
\end{proposition}
Each $\sigma$-conjugate $V_{1,\sigma^i}$ is isomorphic to $V_1$ as
$\Q_p$-representation. Thus each tensor factor in the above
decomposition is potentially crystalline as well.
\begin{proof}
Assume $r=2$ for simplicity. The above tensor decomposition
implies the tensor decomposition of $\C_p$-representations:
$$
V_{\Q_p}\otimes_{\Q_p}\C_p\cong(V_1\otimes_{\Q_{p^2}}\C_p)\otimes_{\C_p}(V_{1,\sigma}\otimes_{\Q_{p^2}}\C_p).
$$
Since $V_{\Q_p}$ is crystalline, it is Hodge-Tate. It implies that
the Sen's operator $\Theta_{V}$ of $V_{\Q_p}$ is diagonalizable
over $\C_p$. Let $\Theta_{V_1}$ be the Sen's operator of $V_1$. It
can be written naturally as $\Theta_1\oplus\Theta_{1,\sigma}$
where $\Theta_{1}$ is associated with $V_1\otimes_{\Q_{p^2}}\C_p$
and $\Theta_{1,\sigma}$ to $V_{1,\sigma}\otimes_{\Q_{p^2}}\C_p$.
Thus one has
$$
\Theta_V=\Theta_{1}\otimes id +id\otimes \Theta_{1,\sigma}.
$$
It implies that $\Theta_{1}$ and $\Theta_{1,\sigma}$ are
diagonalizable. Now consider the eigenvalues of them. For that we
use the relation between the Hodge-Tate cocharacter and the
eigenvalues of the Sen's operator: they are related by the maps
$\log$ and $\exp$. Continue the argument about Hodge-Tate
cocharacter in Proposition \ref{first crystalline tensor
decomposition}. So let $\{\tau=\tau_1,\tau_2\}$ be the
$\Gal_{\Q_p}$-orbit of $\tau$. We can assume that in the above
decomposition the $V_1$-factor corresponds to $\tau$. It follows
that the projection of $\mu_{HT}$ to the $V_{1,\sigma}$-factor is
trivial. This implies that the eigenvalues of $\Theta_{1,\sigma}$
are zero. Particularly they are integral. So are those of
$\Theta_1$. Hence $\Theta_{V_1}$ is diagonalizable with integral
eigenvalues. So $V_1$ is Hodge-Tate, and by Theorem \ref{HT
implies de-rham } it is de Rham. By the $p$-adic monodromy
theorem, conjectured by Fontaine and firstly proved by Berger (see
\cite{Be}), it is potentially log crystalline. One shows further
that it is potentially crystalline. Let $N_V$ (resp. $N_{V_1}$) be
the monodromy operator of $V$ (resp. $V_1$). Then one has the
formulas:
$$
N_{V_1}=N_1+N_{1,\sigma}, \ N_V=N_{1}\otimes id+ id\otimes
N_{1,\sigma}.
$$
Since $V$ is crystalline, $N_V=0$. It implies that
$N_1=N_{1,\sigma}=0$. Hence $N_{V_1}=0$ and $V_1$ is potentially
crystalline.
\end{proof}

\subsection{Consequence on the possibilities of the Newton polygon}
We show that the admissibility of a filtered $\phi$-module
associated with a crystalline representation yields a classification
of the Newton polygons in $M_0(\bar k)$. The following functors were
introduced by Fontaine in order to study a
$\Q_{p^r}$-representation:
\begin{definition}
Let $V$ be a $\Q_{p^r}$-representation. For each $0\leq m\leq
r-1$, one defines
$$
D_{crys,r}^{(m)}(V):=(V\otimes_{\Q_{p^r},\sigma^m}B_{crys}
)^{\Gal_E}.
$$
\end{definition}
One defines similarly the functors $\{D_{dR,r}^{(m)}(V)\}_{0\leq
m\leq r-1}$ by replacing $B_{crys}$ with $B_{dR}$. The following
lemma is obvious:
\begin{lemma}\label{restriction of scalar of V}
Let $V$ be a $\Q_{p^r}$-representation. Then there is a natural
isomorphism of $\Gal_{E}$-representations
$$
V\otimes_{\Q_p}\Q_{p^r}\cong\oplus_{m=1}^{r-1}V\otimes_{\Q_{p^r},\sigma^m}\Q_{p^r}.
$$
\end{lemma}
By the lemma there is a natural direct decomposition
$$
V\otimes_{\Q_p}B_{crys}\cong
V\otimes_{\Q_p}(\Q_{p^r}\otimes_{\Q_{p^r}}B_{crys})\cong
(V\otimes_{\Q_p}\Q_{p^r})\otimes_{\Q_{p^r}}B_{crys}\cong\oplus_{m=0}^{r-1}V\otimes_{\Q_{p^r},\sigma^m}B_{crys},
$$
which implies in particular a direct decomposition of $E_0$-vector
spaces
$$
D_{crys}(V)=\oplus_{m=0}^{r-1}D_{crys,r}^{(m)}(V).
$$
It is clear that $V$ is crystalline iff
$\dim_{E_0}D_{crys,r}^{(m)}(V)=\dim_{\Q_{p^r}}V$ for either $m$
holds. Let $V$ be a crystalline $\Q_{p^r}$-representation. Over
$D_{crys}(V)$ there is a natural $\sigma$-linear map $\phi$, and
over $D_{dR}(V)=D_{crys}(V)\otimes_{E_0}E$ there is a natural
filtration $Fil$. We want to study some properties of the
restrictions of them to a direct factor.
\begin{lemma}\label{phi-module and direct sum}
The map $\phi$ permutes the direct factors
$\{D_{crys,r}^{(m)}(V)\}$ cyclically. Consequently, one has the
decomposition of $\phi^r$-modules
$$
(D_{crys}(V),\phi^r)=\oplus_{m=0}^{r-1}(D_{crys,r}^{(m)}(V),\phi^r|_{D_{crys,r}^{(m)}(V)}).
$$
Moreover, each $\phi^r$-submodule
$(D_{crys,r}^{(m)}(V),\phi^r|_{D_{crys,r}^{(m)}(V)})$ has the same
Newton slopes.
\end{lemma}
\begin{proof}
For $d=v\otimes_{\sigma^m}b\in D_{crys,r}^{(m)}(V)$, it follows
from the formula $ \phi(d)=v\otimes_{\sigma^{m+1\mod r}}\phi(b)
 $ that
$\phi(d)\in D_{crys,r}^{(m+1 \mod r)}(V)$. So $\phi$ permutes the
direct factors in a cyclic way. Thus $\phi^r$ is preserved under
the direct decomposition. The last statement can be seen as follows: take
a basis $e$ of $D_{crys,r}^{(0)}(V)$. Then as $\phi$ is a semilinear isomorphism,
$e_m=\phi^{m}(e)$ is a basis of $D_{crys,r}^{(m)}(V)$. Let $A$ be the matrix satisfying
$$
\phi(e_{r-1})=Ae_0.
$$
Then under the basis $\{e_0,e_1,\cdots,e_{r-1}\}$ of $D_{crys}(V)$, the representation matrix of $\phi^r$ reads:
$$
\phi^r\left(
        \begin{array}{c}
          e_0 \\
          e_1 \\
          \vdots\\
          e_{r-1} \\
        \end{array}\right)=
       \left(
         \begin{array}{cccc}
           A & 0 & \cdots & 0 \\
           0 & A^\sigma & \cdots & 0 \\
           \vdots & \vdots & \ddots & \vdots \\
           0 & 0 & \cdots & A^{\sigma^{r-1}} \\
         \end{array}
       \right)\left(
        \begin{array}{c}
          e_0 \\
          e_1 \\
          \vdots\\
          e_{r-1} \\
        \end{array}\right).
$$
From here, one sees the equality of Newton slopes on each factor clearly.
\end{proof}
As a consequence, one can define on the tensor product
$\otimes_{m=0}^{r-1}D_{crys,r}^{(m)}(V)$ a $\phi$-module structure:
for a vector of form $v_0\otimes \cdots\otimes v_{r-1}$, define
$$
\phi_{ten}(v_0\otimes \cdots\otimes v_{r-1}):=\phi(v_{r-1})\otimes
\phi(v_0)\otimes\cdots\otimes \phi(v_{r-2}).
$$
It is easily seen that the $\sigma^r$-linear map $\phi_{ten}^{r}$
is the tensor product of $\phi^r|_{D_{crys,r}^{(m)}(V)}$s. Next we
consider the induced filtration $Fil_m^i:=Fil^i\cap
D_{dR,r}^{(m)}(V)$ on each direct factor $D_{dR,r}^{(m)}(V)$ from
$D_{dR}(V)$. As filtered modules it holds that
 $$
(D_{dR}(V),Fil)=\oplus_{m=0}^{r-1}(D_{dR,r}^{(m)}(V),Fil_m).
 $$
The tensor product $\otimes_{m=0}^{r-1}D_{dR,r}^{(m)}(V)$ is
equipped with the filtration $Fil_{ten}$ which is the tensor
product of $Fil_m$s.

\begin{proposition}\label{structure of D_crys}
Let $V_1$ be a crystalline $\Q_{p^r}$-representation and $V$ a
$\Q_p$ representation such that there is an isomorphism of
$\Q_{p^r}$-representations
$$
V\otimes_{\Q_p} \Q_{p^r}\cong
V_1\otimes_{\Q_{p^r}}V_{1,\sigma}\otimes_{\Q_{p^r}}\cdots\otimes_{\Q_{p^r}}V_{1,\sigma^{r-1}}.
$$
Then there is an isomorphism of filtered $\phi$-modules:
$$
D_{crys}(V)\cong\otimes_{m=0}^{r-1}D_{crys,r}^{(m)}(V_1),
$$
where the filtered $\phi$-module structure on $\otimes_{m=0}^{r-1}D_{crys,r}^{(m)}(V_1)$ is given by $Fil_{ten}$ and $\phi_{ten}$.
\end{proposition}
\begin{proof}
The original proof is lengthy. The current argument is suggested
by the referee. It is simpler and clearer. One observes the following
natural isomorphisms of filtered $\phi$-modules:
\begin{eqnarray*}
  D_{crys}(V)&=&  [V\otimes_{\Q_{p}}B_{crys}]^{\Gal_E} \\
   &\cong& [V\otimes_{\Q_p}(\Q_{p^r}
 \otimes_{\Q_{p^r}}B_{crys})]^{\Gal_E} \\
   &\cong &[(V\otimes_{\Q_p}\Q_{p^r})
 \otimes_{\Q_{p^r}}B_{crys}]^{\Gal_E}\\
 &\cong&  [(V_1\otimes_{\Q_{p^r}}V_{1,\sigma}\otimes_{\Q_{p^r}}\cdots\otimes_{\Q_{p^r}}V_{1,\sigma^{r-1}})
 \otimes_{\Q_{p^r}}B_{crys}]^{\Gal_E}\\
 &\cong& [\bigotimes_{m=0}^{r-1}(V_{1,\sigma^m}\otimes_{\Q_{p^r}}B_{crys})]^{\Gal_E}\\
 &\cong&  [\bigotimes_{m=0}^{r-1}(V_{1}\otimes_{\Q_{p^r},\sigma^m}B_{crys})]^{\Gal_E}\\
 &=&\otimes_{m=0}^{r-1}D_{crys,r}^{(m)}(V_1).
\end{eqnarray*}
An extra explanation is possibly necessary: note first that the $\phi$ action on $B_{crys}$ preserves $\Q_{p^r}$ and acts on it by $\sigma$. So it permutes the terms in the first tensor product in the fourth line of the above isomorphisms. It implies that the resulting $\phi$-structure on the tensor product in the last line as given by $\phi_{ten}$. As $V_1$ is crystalline, the subspace $\bigotimes_{m=0}^{r-1}D_{crys,r}^{(m)}(V_1)$ in the $\Gal_E$-invariant space has the same dimension as $D_{crys}(V)$. Therefore, the equality in the last line follows.
\end{proof}
In the previous proposition we consider the case where $V$ is
polarisable and of Hodge-Tate weights $\{0,1\}$. Here $V$ being
polarisable means that there is a perfect $\Gal_E$-pairing
$V\otimes V\to \Q_p(-1)$. This condition implies that if $\lambda$
is a Newton (resp. Hodge) slopes of $V$, then $1-\lambda$ is also
a Newton (resp. Hodge) slopes of $V$ with the \emph{same}
multiplicity.
\begin{proposition}\label{possible newton slopes}
Let $V$ be a polarisable crystalline representation with
Hodge-Tate weights $\{0,1\}$. If there exists a two dimensional
crystalline $\Q_{p^r}$-representation $V_1$ such that
$$
V\otimes_{\Q_p} \Q_{p^r}\cong
V_1\otimes_{\Q_{p^r}}V_{1,\sigma}\otimes_{\Q_{p^r}}\cdots\otimes_{\Q_{p^r}}V_{1,\sigma^{r-1}}
$$
holds, then it holds that
\begin{itemize}
    \item [(i)] the Hodge slopes of $V_1$ is $\{2r-1\times 0,
1\times 1\}$,
    \item [(ii)] the Newton slopes of $V_1$ is either
$\{2r\times \frac{1}{2r}\}$ or $\{r\times 0,r\times
\frac{1}{r}\}$.
\end{itemize}
Consequently, there are only two possible Newton slopes for $V$:
$\{2^r\times \frac{1}{2}\}$ or $\{1\times 0, \cdots, {r\choose
i}\times \frac{i}{r},\cdots, 1\times 1\}$.
\end{proposition}
\begin{proof}
Since the Hodge slopes of $V$ are $\{n\times 0 , n\times 1\}$, by
Proposition \ref{structure of D_crys}, there exists a
unique factor $D^{(i)}_{crys,r}(V_1)$ with two distinct Hodge
slopes $\{0,1\}$ and the other factors have all Hodge slopes zero.
Without loss of generality one can assume that
$D^{(0)}_{crys,r}(V_1)$ has Hodge slopes $\{1\times 0,1\times 1\}$
(and any other factor has $\{2\times 0\}$). Summing up the Hodge
slopes of all factors, one obtains the Hodge polygon of
$D_{crys}(V_1)$ as claimed. By the admissibility of the filtered
$\phi$-module structure on $D_{crys}(V_1)$, one finds that its
Newton slopes must be of form $\{m_1\times 0, m_2\times \lambda\}$
where $m_1+m_2=2r$ holds, and $\lambda\in \Q$ satisfies $\lambda
m_2=1$. By Lemma \ref{phi-module and direct sum}, one finds that
$r|m_i,\ i=1,2$. So $\frac{m_1}{r}+\frac{m_2}{r}=2$ and $m_2\neq
0$. There are only two possible cases:

{\bf Case 1}: $m_1=0$. It implies that $m_2=2r$ and
$\lambda=\frac{1}{2r}$.

{\bf Case 2:} $m_1\neq 0$. It implies that $m_1=m_2=r$ and
$\lambda=\frac{1}{r}$.
\end{proof}

Now we can prove the following
\begin{theorem}\label{classification and existence of newton polygon}
Notation as above. Then there are two possible Newton polygons in
$M_0(\bar k)$. Precisely it is either $\{2^{d+\epsilon(D)}\times
\frac{1}{2}\}$ or
$$
\{2^{d-r+\epsilon(D)}\times 0, \cdots, 2^{d-r+\epsilon(D)}\cdot{r
\choose i}\times \frac{i}{r}, \cdots, 2^{d-r+\epsilon(D)}\times
1\}.
$$
\end{theorem}
\begin{proof}
Let $x_0\in M_0(\bar k)$, $A_{x_0}$ and $A_x$ as above. The question
is to determine the possible Newton polygons of the filtered
$\phi$-module $D_{crys}(H^1_{et}(\bar{A}_{x^0},\Q_p))$. For a
suitable large finite extension $E$ of $\Q_p$, we can assume the
direct-tensor decomposition of the $\Gal_E$-module
$H^1_{et}(\bar{A}_{x^0},\Q_p)$ as given in Corollary \ref{tensor
decomposition of galois representation}. Now Propositions \ref{first
crystalline tensor decomposition} and \ref{second crystalline tensor
decomposition} imply that we can even assume that each tensor factor
in the tensor decomposition is crystalline. An unramified factor
contributes only to a multiplicity in the Newton polygon. Hence the
theorem follows directly from Proposition \ref{possible newton
slopes}.
\end{proof}

\begin{remark}\label{remark on the existence of NP}
In this remark we would like to discuss the existence of each Newton
polygon in the classification. To this end, it suffices to realize
that the method of Noot for the original example of Mumford (see
\S3-5 \cite{Noot2}) can be generalized directly: Noot studied the
reductions of CM points of a Mumford's family. The set of CM points
can be divided into two types: let $F\subset J$ be a maximal
subfield of $D$. Then $J$ can either be written in a form
$N\otimes_{\Q}F$ ($N$ is necessary an imaginary quadratic extension
of $\Q$) or not in such a form. To our purpose one finds that the
second case generalizes, and the resulting generalization gives the
necessary existence result. More precisely, Proposition 5.2 loc.
cit. provides the maximal subfields in $D$ of the second type with
the following freedom: Let $\mathfrak{p}$ be a prime of $F$ over
$p$. Then $J$ can be so chosen that $\mathfrak p$ is split or inert
in $J$. Secondly, Lemma 3.5 and Proposition 3.7 loc. cit. work
verbatim for a general $D$ except that in the case $\epsilon(D)=1$
one adds the multiplicity two to the constructions appeared therein.
This step gives us an isogeny class of CM abelian varieties which
appear as $\bar \Q$-points of $M_K$, and also as $\bar \Z_p$-points
of $M$ and hence in $M_0(\bar k)$. Finally the proof of Proposition
4.4 loc. cit., namely the method of computing the Newton polygon for
a CM abelian variety modulo $p$, works in general. Thus one can also
conclude the existence result for the general case.
\end{remark}

\section{A direct-tensor decomposition of the universal filtered
Dieudonn\'{e} module over a formal neighborhood}\label{Section on
formal tensor decomposition} Let $x_0$ be a $k$-rational point of
$M$ and $\hat{M}_{x_0}$ the completion of $M$ at $x_0$. The aim of
this section is to show a direct-tensor decomposition of the
restriction of $(H,F,\nabla,\phi)$ to the formal neighborhood of
$x_0$. Let $E$ be a finite extension of $L$ and $x^0$ an
$E$-rational point of $M$ which specializes into $x_0$. Corollary
\ref{tensor decomposition of galois representation} gives a
decomposition of $H_{\Z_p}$ into direct-tensor product of
$\Gal_E$-lattices. By Propositions \ref{second crystalline tensor
decomposition} and \ref{possible newton slopes}, $V_1\subset
V_1\otimes \Q_p$ is a $\Gal_E$-lattice of a two dimensional
potentially crystalline $\Q_{p^r}$-representation with Hodge-Tate
weights $\{2r-1\times 0, 1\times 1\}$. By making a possible finite
extension of $E$, we can assume that $V_1\otimes \Q_p$ is already
crystalline as $\Gal_E$-module.

\subsection{Drinfel'd's $\sO_{\mathfrak p}$-divisible module and versal
deformation}\label{Drinfel'd module and versal deformation
subsection} Recall the following notion of $\sO_{\mathfrak
p}$-divisible modules due to Drinfel'd (see Appendix \cite{Car}):
\begin{definition}
Let $S=\Spec R$ be an $\sO_{\mathfrak p}$-scheme. An
$\sO_{\mathfrak p}$-divisible module over $S$ is a pair $(G,f)$
consisting of a $p$-divisible group $G$ over $S$ and an action of
$\sO_{\mathfrak p}$ on $G$:
$$
f: \sO_{\mathfrak p}\to \End(G)
$$
satisfying
\begin{itemize}
    \item [(a)] $f(1)$ is the identity,
    \item [(b)] $G^0$ (the connected part of $G$) is of dimension 1,
    \item [(c)] the derivation of $f$, $f': \sO_{\mathfrak p}\to
    \End(Lie(G))=R$ coincides with the structural morphism
    $\sO_{\mathfrak p}\to R$.
\end{itemize}
\end{definition}
By the fundamental theorem of C. Breuil (see Corollary 3.2.4,
Theorem 3.2.5 \cite{Br}), the $\Gal_E$-lattice $V_1$ corresponds to
a $p$-divisible group $B$ over $\sO_E$.
\begin{lemma}\label{introduction of dieudonne module}
The corresponding $p$-divisible group $B$ is an $\sO_{\mathfrak
p}$-divisible module over $\sO_E$ of height 2.
\end{lemma}
\begin{proof}
By construction, one has the inclusion $\Z_{p^r}\subset
End_{\Gal_E}(V_1)$. Since the functor of Breuil is an
anti-equivalence of categories, one obtains an inclusion as well
 $
\sO_{\mathfrak p}\cong\Z_{p^r}\subset End(B)
 $.
The condition (a) is obvious. The condition (b) on the dimension
of $G$ and the assertion on the height of $B$ follow from the
Hodge-Tate weights of $V_1\otimes \Q_p$ given in Proposition
\ref{possible newton slopes}. By taking the derivation of the
inclusion, one obtains an inclusion of $\Z_p$-algebras
$\sO_{\mathfrak p}\subset \sO_E$, which ought to be the structural
morphism by the naturalness of the functor.
\end{proof}

Let $M_B$ be the filtered Dieudonn\'{e} module associated with
$B$. Denote by $L_B$ the previous $\Gal_E$-lattice $V_1$. Fix a
generator $s$ of $\Z_{p^r}$ as $\Z_p$-algebra. The image of $s$ in
$\End_{\Gal_E}(L_B)\subset \End_{\Z_p}(L_B)\subset L_B^{\otimes}$
is an \'{e}tale Tate cycle of $L_B$ and is denoted by $s_{B,et}$.
By the $p$-adic comparison theorem, one has a natural isomorphism
respecting $\Gal_E$-action, filtrations and $\phi$s:
$$
L_B\otimes_{\Z_p}B_{crys}\cong M_B\otimes_{W(k)}B_{crys}.
$$
By one of the main technical results of \cite{Ki} (see Proposition
0.2 loc. cit.), the crystalline Tate cycle $s_B$, which
corresponds to $s_{B,et}\otimes 1\in L_B\otimes_{\Z_p}B_{crys}$ in
the above comparison, lies actually in $M_B^{\otimes}$. Let
$G_{B}\subset \GL_{\Z_p}(L_B)$ (resp. $\sG_B\subset
\GL_{W(k)}(M_B)$) be the subgroup defined by $s_{B,et}$ (resp.
$s_B$). By Corollary 1.4.3 (4) loc. cit., the filtration
$Fil^1\otimes k$ on $M_B\otimes k$ is
$\sG_B\otimes_{W(k)}k$-split. Choose a cocharacter $\mu_0: \G_m\to
\sG_B\otimes k$ inducing this filtration and further choose a
cocharacter $\mu: \G_m\to \sG_B\subset \GL(M_B)$ lifting $\mu_0$
(see 1.5.4 loc. cit.). The cocharacter $\mu$ defines the opposite
unipotent subgroups:
$$
\xymatrix{
  U_{\sG_B}\ar[d]_{\cap} \ar[r]^{\subset} & \sG_B \ar[d]^{\cap} \\
  U_B \ar[r]^{\subset\quad} & \GL(M_B)   }
$$
By construction, one has $U_{\sG_B}=U_B\cap \sG_B$. Let
$\hat{U}_{\sG_B}$ (resp. $\hat{U}_{B}$) be the completion of
$U_{\sG_B}$ at the identity section of $U_{\sG_B}$ (resp. $U_B$),
whose corresponding complete local rings are denoted by
$R_{\sG_B}$ and $R$ respectively. The filtration on $M_B$ defined
by $\mu$ corresponds to a $p$-divisible group $B'$ over $W(k)$
whose closed fiber $B'\otimes k$ is isomorphic to $B\otimes k$ as
$p$-divisible group over $k$. For a later use we denote this $\mu$
also by $\mu_{B'}$. Write $(M_B', Fil^1_{M_B'},\phi_{M_B'})$ for
the tuple defined by the filtered crystal structure on $M_B'$, and
fix $\phi_{\hat{U}_{\sG_B}}: R_{\sG_B}\to R_{\sG_B}$ a lifting of
the absolute Frobenius. Following Faltings Remarks \S7 \cite{Fa3}
(see also \S1.5 \cite{Ki}), one defines a filtered $F$-crystal
over $\hat{U}_{\sG_B}$ by the tuple
$$
(\sN_B=M_B'\otimes_{W(k)}R_{\sG_B},
Fil^1_{\sN_B}=Fil^1_{M_B'}\otimes_{W(k)}R_{\sG_B},
\phi_{\sN_B}=u\circ(\phi_{M_B'}\otimes \phi_{\hat{U}_{\sG_B}})),
$$
where $u\in U_{\sG_B}(R_{\sG_B})$ is the tautological
$R_{\sG_B}$-point of $U_{\sG_B}$. Then by Faltings loc. cit., there
is a unique integrable connection $\nabla_{\sN_B}$ over $\sN_B$ such
that the quadruple
$(\sN_B,Fil^1_{\sN_B},\phi_{\sN_B},\nabla_{\sN_B})$ defines an
object in $\mathcal{MF}_{[0,1]}^{\nabla}(\hat U_{\sG_B})$. By
Faltings Theorem 7.1 \cite{Fa2}, there is a $p$-divisible group
$\sB$ over $R_{\sG_B}$, unique up to isomorphism, such that the
attached filtered Dieudonn\'{e} module to $\sB$ is isomorphic to the
above quadruple. If we replace everything of $U_{\sG_B}$ with that
of $U_B$, the above discussion gives then a versal deformation of
$B\otimes k$ over $\hat U_B$, which by abuse of notation is denoted
again by $\sB$. In this context, the sublocus
$\Spf(R_{\sG_B})\hookrightarrow \Spf(R)$ has an interpretation as
the versal deformation respecting the Tate cycles which are
stabilized by $\sG_B$ (see \S7 \cite{Fa3} and Corollary 1.5.5
\cite{Ki}). By Corollary 1.5.11 \cite{Ki}, $B$ is isomorphic to the
pull-back of $\sB$ along a $W(k)$-algebra morphism $R_{\sG_B}\to
\sO_E$.\\

We proceed to study the natural $\sG_B$-action on $M_B$ via the
inclusion $\sG_B\subset \GL_{W(k)}(M_B)$. Recall that we have fixed
an element $s\in \Z_{p^r}$. Let $\{s_i:=s^{\sigma^i}\}_{0\leq i\leq
r-1}\subset\Z_{p^r}$ be the Galois conjugates of $s$. The minimal
polynomial of $s_B\in \End_{W(k)}(M_B)$ is that of $s\in \Q_{p^r}$
over $\Q_p$. As $\Z_{p^r}\subset W(k)$, the minimal polynomial of
$s_B$ splits into linear factors and one has then the relation in
$\End_{W(k)}(M_B)$:
$$
(s_B-s_0)\cdots(s_B-s_{r-1})=0.
$$
Let $M_i\subset M_B$ be the eigenspace of $s_B$ corresponding to
the eigenvalue $s_i$. Recall that we have shown in \S\ref{Section
on two dimesional crystalline rep} a direct sum decomposition of
$M_B\otimes {\rm Frac}(W(k))=D_{crys}(L_B\otimes \Q_p)$ by using
the functors $D_{crys,r}^{(i)}$ of Fontaine. Now we have the
following
\begin{lemma}\label{lattic structure of fontaine functors}
The eigen decomposition $M_B=\oplus_{i=0}^{r-1}M_i$ is a lattice
decomposition of
$$
M_B\otimes {\rm
Frac}(W(k))=\oplus_{i=0}^{r-1}D_{crys,r}^{(i)}(L_B\otimes \Q_p).
$$
Namely, $M_i$ is a lattice in $D_{crys,r}^{(i)}(L_B\otimes \Q_p)$
for each $i$.
\end{lemma}
\begin{proof}
In the comparison isomorphism $L_B\otimes_{\Z_p}B_{crys}\cong
M_B\otimes_{W(k)}B_{crys}$, the endomorphism $s_{B,et}\otimes 1$
corresponds to $s_B\otimes 1$. Also the endomorphisms commute with
the $\Gal_E$-actions on both sides. The endomorphism
$s_{B,et}\otimes 1$ on $L_B\otimes_{\Z_{p}}\Z_{p^r}$ decomposes
into eigenspaces and so we obtain a $\Z_{p^r}[\Gal_E]$-module
decomposition:
$$
L_B\otimes_{\Z_{p}}B_{crys}=(L_B\otimes_{\Z_{p}}\Z_{p^r})\otimes_{\Z_{p^r}}B_{crys}=\oplus_{i=0}^{r-1}L_i\otimes_{\Z_{p^r}}B_{crys},
$$
where $L_i$ is the $\Z_{p^r}$-submodule of
$L_B\otimes_{\Z_{p}}\Z_{p^r}$ corresponding to the eigenvalue
$s_i$. Under the comparison isomorphism it corresponds to the
decomposition of $W(k)[\Gal_E]$-module:
$$
M_B\otimes_{W(k)}B_{crys}=\oplus_{i=0}^{r-1}M_i\otimes_{W(k)}B_{crys}.
$$
Taking the $\Gal_E$-invariants, we obtain
$$
(L_i\otimes_{\Z_{p^r}}B_{crys})^{\Gal_E}=M_i\otimes_{W(k)}{\rm
Frac}(W(k)).
$$
Finally one notices that the two indexed sets of $\Gal_E$-modules
$\{L_i\otimes_{\Z_{p^r}}B_{crys}\}_{0\leq i\leq r-1}$ and
$\{L_B[\frac{1}{p}]\otimes_{\Q_{p^r,\sigma^i}}B_{crys}\}_{0\leq
i\leq r-1}$ are actually equal. The lemma follows.
\end{proof}

\begin{proposition}\label{tensor factor has a filtered phi module structure}
The tensor product $\otimes_{i=0}^{r-1}M_i$ is a lattice of the
admissible filtered $\phi$-module
$\otimes_{i=0}^{r-1}D_{crys,r}^{(i)}(L_B\otimes \Q_p)$ in
Proposition \ref{structure of D_crys}.
\end{proposition}
\begin{proof}
The filtration $Fil^1_{B}$ on $M_B\otimes_{W(k)}\sO_E$ is filtered
free and restricts to the filtration $Fil^1_i$ on each direct
factor $M_i\otimes \sO_E$. Also one sees from the proof of
Proposition \ref{possible newton slopes} that there is a unique
factor $M_i$ with nontrivial $Fil^1_i$. Since $\phi_{M_B'}$ is
$\sigma$-linear, it permutes the eigen factors $\{M_i\}_{0\leq
i\leq r-1}$ cyclically. The proposition is now clear.
\end{proof}
For a later use, we denote the above lattice by
$(\otimes_{i=0}^{r-1}M_i,Fil^1_{ten},\phi_{ten})$.
\begin{lemma}\label{the action of group on eigenfactors}
The eigen decomposition of $M_B$ is also a decomposition as
$\sG_B$-module. In fact, the $W(k)$-group $\sG_B$ is naturally
isomorphic to $\prod_{i=0}^{r-1}\GL_2(W(k))$, and the
$\sG_B$-module $M_B$ is isomorphic to the
$\prod_{i=0}^{r-1}\GL_2(W(k))$-module
$\oplus_{i=0}^{r-1}(W(k)^{\oplus 2})_i$, in which the $i$-th
factor $(W(k)^{\oplus 2})_i$ is the tensor product of the standard
representation of the $i$-th factor $\GL_2(W(k))$ and the trivial
representations of the $j$-th factors with $j\neq i$.
\end{lemma}
\begin{proof}
Because the $\sG_B$-action on $M_B$ commutes with the $s_B$-action
by definition, the eigen decomposition of $M_B$ with respect to
$s_B$ is preserved by the $\sG_B$-action. This can be seen more
clearly if we go to the \'{e}tale side: first of all, it is easy
to see that the commutant subalgebra of $\Z_{p^r}\subset
\End_{\Z_p}(L_B)\cong M_{2r}(\Z_p)$ is $\End_{\Z_{p^r}}(L_B)\cong
M_{2}(\Z_{p^r})$. So the group $G_B\subset
\GL_{\Z_p}(L_B)\cong\GL_{2r}(\Z_p)$ is isomorphic to
$\GL_2(\Z_{p^r})$, and particularly it is connected. Next, by
Corollary 1.4.3 (3) \cite{Ki}, there is a $W(k)$-linear
isomorphism $L_B\otimes_{\Z_p}W(k)\cong M_B$ which induces an
isomorphism
 $
G_{B}\times_{\Z_p}W(k)\cong\sG_B
 $.
As $\Z_{p^r}\subset W(k)$, one has isomorphisms
$$
\sG_B\cong\GL_2(\Z_{p^r}\otimes_{\Z_p}W(k))\cong\GL_2(\prod_{i=0}^{r-1}W(k))=\prod_{i=0}^{r-1}\GL_2(W(k)).
$$
Under the above isomorphism, the $\sG_B$-module $M_B$ is isomorphic
to the $G_{B}$-module $L_B$ tensorizing with $W(k)$. Moreover the
isomorphism preserves the eigen decompositions of both sides. As
said above, the $G_B$-action on $L_B$ is isomorphic to the standard
representation of $\GL_2(\Z_{p^r})$ on $\Z_{p^r}^{\oplus 2}$, which
is considered as a $\Z_{p}$-group acting on a $\Z_{p}$-module by
restriction of scalar. Thus, the action after tensorizing with
$\Z_{p^r}$, that is the standard
$\GL_2(\Z_{p^r}\otimes_{\Z_p}\Z_{p^r})$-action on
$(\Z_{p^r}\otimes_{\Z_p}\Z_{p^r})^{\oplus 2}$, splits: write
$$
\GL_2(\Z_{p^r}\otimes_{\Z_p}\Z_{p^r})\cong\prod_{i=0}^{r-1}\GL_2(\Z_{p^r}),
\ g\mapsto (g_0,\cdots,g_{r-1}),
$$
and
$$
(\Z_{p^r}\otimes_{\Z_p}\Z_{p^r})^{\oplus
2}\cong\prod_{i=0}^{r-1}\Z_{p^r}^{\oplus 2}, \ v\mapsto
(v_0,\cdots,v_{r-1}).
$$
Then $g(v)$ is mapped to $(g_0v_0,\cdots,g_{r-1}v_{r-1})$. So does
the action after tensorizing with the larger ring $W(k)$. Hence the
lemma follows.
\end{proof}
\begin{proposition}
The sublocus $\hat{U}_{\sG_B}$ is a versal deformation of the
$p$-divisible group $B\otimes k$ as $\sO_{\mathfrak p}$-divisible
module.
\end{proposition}
\begin{proof}
We calculate first the dimension of $\hat{U}_{\sG_B}$. It is equal
to $\rank_{W(k)}\frac{\mathfrak g}{Fil^0\mathfrak g}$, where
$\mathfrak{g}$ is the Lie algebra of $\sG_B$ and the filtration is
the restriction of the tensor filtration on
$\End_{W(k)}(M_B)=M_B^\vee\otimes M_B$ via the inclusion $\mathfrak
g\subset \End_{W(k)}(M_B)$. We claim that it is one dimensional.
By the discussion on the filtration in the proof of Proposition
\ref{tensor factor has a filtered phi module structure} and Lemma
\ref{the action of group on eigenfactors}, one has an isomorphism
of Lie algebras over $W(k)$
 $
\mathfrak g\cong\oplus_{i=0}^{r-1}\mathfrak{gl}_2
 $,
such that there is a unique factor with the nontrivial induced
filtration through the isomorphism. This shows the claim. Now let
$\tilde{B}\to Z$ be a versal deformation of $B\otimes k$ as
$\sO_{\mathfrak p}$-divisible modules. Thus one has a map $\tilde
f: \sO_{\mathfrak p}\to \End(\tilde B)$ which makes the
commutative diagram
$$
\xymatrix{
 \sO_{\mathfrak p} \ar[dr]_{f} \ar[r]^{\tilde f}
                & \End(\tilde B)\ar[d]^{\otimes k}  \\
                & \End(B)             }
$$
Let $s_{cycle}\in \End(B)$ and $\tilde{s}_{cycle}\in \End(\tilde
B)$ be the images of $s\in \sO_{\mathfrak p}$ in the endomorphism
$\Z_p$-algebras. The element $s_{cycle}$ corresponds to $s_B$
under the Dieudonn\'{e} functor, and by Faltings Theorem 7.1
\cite{Fa2}, the corresponding element $\tilde s_B$ to $\tilde
s_{cycle}$, as an endomorphism of the filtered Dieudonn\'{e}
crystal attached to $\tilde B$, is a crystalline Tate cycle and is
the parallel continuation of $s_B$ over $Z$. By the universal
property Proposition 4.9 \cite{Mo}, the inclusion $Z\subset
\hat{U}_{_B}$ factors
 $ Z\subset
\hat{U}_{\sG_B}\subset \hat{U}_{_B}
 $.
As both $Z$ and $\hat{U}_{\sG_B}$ are formally smooth of dimension
one, it follows that $Z=\hat{U}_{\sG_B}$.
\end{proof}
Let $(\sN_B,Fil^1_{\sN_B},\phi_{\sN_B}, \nabla_{\sN_B})$ be the
universal filtered Dieudonn\'{e} module attached to $\sB$ over
$\hat{U}_{\sG_B}$. Put $\sN_i=M_i\otimes R_{\sG_B}, 0\leq i\leq
r-1$. Then the Tate cycle $s_B\in \End_{R_{\sG_B}}(\sN_B)$ induces
the eigen decomposition
$$
(\sN_B,Fil^1_{\sN_B},\nabla_{\sN_B})=\oplus_{i=0}^{r-1}(\sN_i,Fil^1_{\sN_i},
\nabla_{\sN_i}),
$$
where $Fil^1_{\sN_i}$ (resp. $\nabla_{\sN_i}$) is the restriction of
$Fil^1_{\sN_B}$ (resp. $\nabla_{\sN_B}$) to $\sN_i$. However the
eigen decomposition is not preserved by $\phi_{\sN_B}$: recall that
$\phi_{\sN_B}=u\circ(\phi_{M_B'}\otimes \phi_{\hat{U}_{\sG_B}})$. As
$U_{\sG_B}\subset \sG_B$, $u$ preserves the eigen decomposition by
Lemma \ref{the action of group on eigenfactors}. So $\phi_{\sN_B}$
permutes the factors in the eigen decomposition in a cyclic way. In
order to state the following decomposition result, we need to
introduce the category $\mathcal{MF}_{big,r}^{\nabla}$, which is
analogous to the category $\mathcal{MF}_{big}^{\nabla}$ introduced
by Faltings (see c)-d) Ch. II \cite{Fa2}). The category
$\mathcal{MF}_{big,r}^{\nabla}(R_{\sG_{B}})$ consists of four tuples
$(N,Fil,\phi_r,\nabla)$, where $N$ is a free $R_{\sG_B}$-module,
$Fil$ is a sequence of $R_{\sG_B}$-submodules with $Gr_{Fil}(N)$
torsion free,
$$
\phi_r: N\otimes_{R_{\sG_B},\phi_{\hat{U}_{\sG_B}}^r}R_{\sG_B}\to
N
$$ satisfies the divisibility condition $\phi_r(Fil^i)\subset p^{i}N$,
and $\nabla$ an integrable connection satisfying the Griffiths
transversality and finally $\phi_r$ is parallel with respect to
$\nabla$. Summarizing the above discussions, we have shown the
\begin{proposition}\label{claim 1}
The object $(\sN_B,Fil^1_{\sN_B},\nabla_{\sN_B})$ has a
decomposition
$$
(\sN_B,Fil^1_{\sN_B},\nabla_{\sN_B})=\oplus_{i=0}^{r-1}(\sN_i,Fil^1_{\sN_i},
\nabla_{\sN_i}),
$$
such that $\phi_{\sN_B}$ permutes the factors cyclically.
Consequently, one has a direct sum decomposition in the category
$\mathcal{MF}_{big,r}^{\nabla}(\hat{U}_{\sG_B})$:
$$
(\sN_B,Fil^1_{\sN_B},\phi_{\sN_B}^r,\nabla_{\sN_B})=\bigoplus_{i=0}^{r-1}(\sN_i,Fil^1_{\sN_i},
\phi_{\sN_i},\nabla_{\sN_i}),
$$
where $\phi_{\sN_i}$ is the restriction of $\phi_{\sN_B}$ to
$\sN_i$.
\end{proposition}
As a consequence, one can define an object in
$\mathcal{MF}_{[0,1]}^{\nabla}(\hat{U}_{\sG_B})$ by equipping the
tensor product
$\otimes_{i=0}^{r-1}(\sN_B,Fil^1_{\sN_B},\nabla_{\sN_B})$ with the
Frobenius $\phi_{ten}$, a construction mimic to Proposition
\ref{tensor factor has a filtered phi module structure}.
\subsection{Tensor decomposition of the universal filtered Dieudonn\'{e} module over a formal
neighborhood}\label{subsection on formal tensor decomposition}
Notation as the introductory part of the section. Let $A$ be the
abelian scheme over $\sO_E$ with the closed fiber (resp. generic
fiber) $A_0$ (resp. $A^0$) given by $x_0$ (resp $x^0$). Put
$L_A=H_{\Z_p}=H^1_{et}(\bar{A^0},\Z_p)$. For simplicity of notation,
we use the same letters $A$ etc. to mean the associated
$p$-divisible groups. Recall that Corollary \ref{tensor
decomposition of galois representation} gives a $\Gal_E$-lattice
decomposition $ L_A=(V_{\Z_p}\otimes U_{\Z_p})^{\oplus
2^{\epsilon(D)}}
 $.
Let $A_1$ and $A_2$ be the two $p$-divisible groups over $\sO_E$
corresponding to the lattice $V_{\Z_p}$ and $U_{\Z_p}$
respectively by the theorem of Breuil (see \cite{Br}). Write
$L_{A_1}=V_{\Z_p}$ and $L_{A_2}=U_{\Z_p}$. Let
$(M_A,Fil^1_A,\phi_A)$ be the filtered Dieudonn\'{e} module
attached to $A$. Similar notations for $A_1$ and $A_2$.
\begin{proposition}\label{structure of D_crys in lattice level}
One has a natural isomorphism of filtered $\phi$-modules:
$$
(M_A,Fil^1_A,\phi_A)\cong[(M_{A_1},Fil^1_{A_1},\phi_{A_1})\otimes(M_{A_2},Fil^1_{A_2},\phi_{A_2})
]^{\oplus 2^{\epsilon (D)}},
$$
where the factor $(M_{A_1},Fil^1_{A_1},\phi_{A_1})$ is naturally
isomorphic to $(\otimes_{i=0}^{r-1}M_i,Fil^1_{ten},\phi_{ten})$ in
Proposition \ref{tensor factor has a filtered phi module
structure}, and the factor $(M_{A_2},Fil^1_{A_2},\phi_{A_2})$ is a
unit crystal.
\end{proposition}
\begin{proof}
We have shown the above isomorphisms after inverting $p$ of both
sides: the first isomorphism is a consequence of Propositions
\ref{first crystalline tensor decomposition}, \ref{second
crystalline tensor decomposition} as the functor $D_{crys}$ commutes
with tensor product. The second isomorphism is Proposition
\ref{structure of D_crys}. Also since $U_{\Q_p}$ is an unramified
$\Gal_E$-representation, $(M_{A_2},\phi_{A_2})$ is a unit crystal
and the filtration $Fil^1_{A_2}$ is trivial. To show that the
isomorphisms hold without inverting $p$, we shall apply the theory
of $\mathfrak S$-modules of Kisin developed in \cite{Ki0} and \S1.2,
\S1.4 \cite{Ki}. Consider the first isomorphism: apply first the
functor $\mathfrak M$ to the $\Gal_E$-lattice decomposition
 $
L_A=(L_{A_1}\oplus L_{A_2})^{\oplus 2^{\epsilon(D)}}
 $.
From the proof of Theorem 1.2.2 \cite{Ki} (see 1.2.2 loc. cit.),
one sees that the functor $\mathfrak M$ respects the tensor
product. So after this step one obtains a corresponding
decomposition of $\mathfrak S$-modules. To get the decomposition
of the filtered Dieudonn\'{e} modules as claimed in the first
isomorphism, one applies next Theorem 1.4.2 and Corollary 1.4.3
(i) loc. cit. to each factor in the previous decomposition of
$\mathfrak S$-modules. Consider then the second isomorphism: by
Corollary \ref{tensor decomposition of galois representation}, one
has a tensor decomposition of $\Z_{p^r}[\Gal_E]$-modules:
 $
L_{A_1}\otimes_{\Z_p}\Z_{p^r}=\bigotimes_{i=0}^{r-1}L_{B,\sigma^i}
 $,
where $L_{B,\sigma^i}=L_B\otimes_{\Z_{p^r},\sigma^i}\Z_{p^r}$ ,
which is also equal to $L_i$ in the eigen decomposition of
$L_B\otimes_{\Z_p} \Z_{p^r}$ with respect to $s_{B,et}$ in the
proof of Lemma \ref{lattic structure of fontaine functors}. Taking
the $r$-th tensor power of
 $
L_B\otimes_{\Z_p}\Z_{p^r}=\oplus_{i=0}^{r-1}L_i
 $,
and then the $\Gal(\Z_{p^r}|\Z_p)$-invariants of both sides, one
gets a natural direct decomposition
 $
L_B^{\otimes r}=L_{A_1}\oplus L_{A_1}'
 $ of $\Z_p[\Gal_E]$-modules.
Thus $L_{A_1}$ is naturally isomorphic to a $\Gal_E$-sublattice of
$L_B^{\otimes r}$. Proposition \ref{structure of D_crys} shows that
this sublattice is in fact of Hodge-Tate weights $\{0,1\}$ with the
induced filtered $\phi$-module structure given by Proposition
\ref{tensor factor has a filtered phi module structure}. This
implies the second isomorphism.
\end{proof}
Put $G_{A}=G_{\Z_p}$ and $\sG_A=G_A\times_{\Z_p}W(k)\subset
\GL_{W(k)}(M_A)$, the subgroup defined by the corresponding
crystalline Tate cycles. Recall that after a conjugation by an
element in $G(\Q_p)$, there is a central isogeny
 $
\Z_p^*\times \prod_{i=1}^{n}\SL_2(\sO_{F_{\mathfrak p_i}})\to G_A
 $.
Let $G_{A_1}$ (resp. $G_{A_2}$) be the image of $\Z_p^*\times
\SL_2(\sO_{F_{\mathfrak p_1}})$ (resp.
$\prod_{i=2}^{n}\SL_2(\sO_{F_{\mathfrak p_i}})$) in $\GL(L_{A_1})$
(resp. $\SL(L_{A_2})$). By the construction of the tensor
decomposition, one has the following commutative diagram:
$$
\xymatrix{\Gal_E\ar[r]^{}\ar[dr]^{}
&G_A\ar[r]^{\twoheadrightarrow}\ar[d]_{}
&G_{A_1}\times G_{A_2}\ar[d]^{}\ar[dl]_{\hookleftarrow}\\
 &\GL(L_{A_1})\times \SL(L_{A_2})
\ar[r]^{\quad \otimes}&\GL(L_{A_1}\otimes L_{A_2}).}
$$
Consider the group homomorphism
$$
\otimes^r: \GL(L_B)\to \GL(L_B^{\otimes r}), g\mapsto (g^{\otimes
r}: v_1\otimes \cdots\otimes v_r\mapsto g(v_1)\otimes\cdots
\otimes g(v_r)).
$$
It is a central isogeny over the image.
\begin{lemma}\label{r-self tensor map}
The restriction of $\otimes^r$ to the subgroup $G_B$ factors
$$
\otimes^r|_{G_B}: G_B\to \GL(L_{A_1})\times
\GL(L_{A_1}')\subset\GL(L_B^{\otimes r}).
$$
\end{lemma}
\begin{proof}
Recall that for a $g\in \GL(L_B)$, $g\in G_{B}$ iff $g(s_B)=s_B$
up to a scalar. This implies that $g\otimes 1$ preserves the eigen
decomposition of $L_B\otimes_{\Z_p}\Z_{p^r}$. So
$\otimes^r(g\otimes 1)$ respects the direct sum decomposition
$$
L_B^{\otimes
r}\otimes_{\Z_p}\Z_{p^r}=L_{A_1}\otimes_{\Z_p}\Z_{p^r}\oplus
L_{A_1}'\otimes_{\Z_p}\Z_{p^r}.
$$
Thus $\otimes^r(g)$ preserves the decomposition $ L_B^{\otimes
r}=L_{A_1}\oplus L_{A_1}'
 $.
Hence the lemma follows.
\end{proof}
Let $\xi_{et}: G_{B}\to \GL(L_{A_1})$ be the composite of
$\otimes^r|_{G_B}$ with the projection to the first factor in the
above lemma. The reductive subgroup $G_{A_1}\subset  \GL(L_{A_1})$
is defined by a finite set of tensors in $L_{A_1}^{\otimes}$.
\begin{lemma}\label{form of tensors for G_A_1}
A tensor in $L_{A_1}^{\otimes n}$ is fixed by $G_{A_1}$ only if
$n=2a$ is even, and it must be of form $\det(L_{A_1})^{\otimes
a}\subset L_{A_1}^{\otimes n}$.
\end{lemma}
\begin{proof}
Assume $n$ positive. The $G_{A_1}$-action respects the tensor
decomposition $
L_{A_1}\otimes_{\Z_{p}}\Z_{p^r}=\otimes_{i=0}^{r-1}L_i
 $. A tensor in $L_{A_1}^{\otimes n}$ is fixed by $G_{A_1}$ is by
definition a rank one $\Z_p$-subrepresentation of $G_{A_1}$. So it
gives rise to a rank one $\Z_{p^r}$-subrepresentation of $G_{A_1}$
in $[\otimes_{i=0}^{r-1}L_i]^{\otimes n}$.
Recall that the $G_{A_1}$-action on $L_i$ is isomorphic to the
$i$-th $\sigma$-conjugate of the standard action of
$\GL_2(\Z_{p^r})$ on $\Z_{p^r}^{\oplus 2}$. Then we study the
$\GL_2(\Q_{p^r})$-invariant lines in $[\otimes_{i=0}^{r-1}(\Q_{p^r}^{\oplus
2})_i]^{\otimes n}$, where $(\Q_{p^r}^{\oplus 2})_i:=\Q_{p^r}^{\oplus
2}\otimes_{\Q_{p^r},\sigma^i}\Q_{p^r}$. For that we apply the standard
finite dimensional representation theory of complex Lie groups (see
\cite{FH} Ch. 6). For a partition $\lambda$ of $n$, one has the
irreducible decomposition of $\prod_{i=0}^{r-1}\GL_2(\Q_{p^r})$-modules:
$$
\SSS_{\lambda}[\otimes_{i=0}^{r-1}(\Q_{p^r}^{\oplus 2})_i]
=\bigoplus_{\lambda_0,\cdots,\lambda_{r-1}}
C_{\lambda_0\cdots\lambda_{r-1}\lambda}\cdot\SSS_{\lambda_0}(\Q_{p^r}^{\oplus
2})_0\otimes \cdots\otimes\SSS_{\lambda_{r-1}}(\Q_{p^r}^{\oplus
2})_{r-1},
$$
where $\lambda_i$ in the summation runs through all possible
partitions of $n$. As $\dim (\Q_{p^r}^{\oplus 2})_{i}=2$, the only
possible $\lambda_i$s are of form $\{n-a,a\}$ for $a\leq
\frac{n}{2}$, and
$$
\SSS_{\{n-a,a\}}(\Q_{p^r}^{\oplus 2})_{i}=\left\{
                                          \begin{array}{ll}
                                            \SSS_{\{n-2a\}}(\Q_{p^r}^{\oplus 2})_{i}=\Sym^{n-2a}(\Q_{p^r}^{\oplus 2})_{i}\otimes [\det(\Q_{p^r}^{\oplus 2})_{i}]^{a}, & \hbox{if\ $2a<n$;} \\
                                            \SSS_{\{a,a\}}(\Q_{p^r}^{\oplus 2})_{i}=[\det(\Q_{p^r}^{\oplus 2})_{i}]^{a}, & \hbox{if\ $2a=n$.}
                                          \end{array}
                                        \right.
$$
Since $\GL_2(\Q_{p^r})$ is embedded into
$\prod_{i=0}^{r-1}\GL_2(\Q_{p^r})$ via $g\mapsto (g,\sigma
g,\cdots,\sigma^{r-1}g)$, the above decomposition is also
irreducible with respect to the $\GL_2(\Q_{p^r})$-action. Summarizing
these discussions, we conclude that there exists a
$G_{A_1}$-invariant tensor $s_\alpha$ in $L_{A_1}^{\otimes n}$
only if $n=2a$ is even, and $s_\alpha\otimes 1\in
[\otimes_{i=0}^{r-1}L_i]^{\otimes n}$ is of form
$\otimes_{i=0}^{r-1}[\det(L_i)]^{a}$, which implies that
$s_\alpha\in \det(L_{A_1})^{\otimes a}$.
\end{proof}
\begin{proposition}
The morphism $\xi_{et}$ factors
$$
\xi_{et}: G_{B}\to G_{A_1}\subset \GL(L_{A_1}),
$$
and the induced morphism $\xi_{et}: G_B\to G_{A_1}$ is a central
isogeny.
\end{proposition}
\begin{proof}
Fix an even natural number $n$. Let $s_{\alpha}\in
L_{A_1}^{\otimes n}$ be a tensor for $G_{A_1}$. It is to show that
the image of $G_B$ under $\xi_{et}$ fixes $s_\alpha$. By Lemma
\ref{form of tensors for G_A_1},  $s_\alpha\otimes
1=\otimes_{i=0}^{r-1}[\det(L_i)]^{\frac{n}{2}} $. It is clear that
for a $g\in G_B$, $\otimes^r(g)\otimes 1$ stabilizes the line
$\otimes_{i=0}^{r-1}[\det(L_i)]^{\frac{n}{2}}$. This implies that
$\otimes^r(g)$ stabilizes $s_{\alpha}$. As $\xi_{et}$ is a central
isogeny over its image and both $G_B$ and $G_{A_1}$ are isomorphic
to $\GL_2(\Z_{p^r})$, $\xi_{et}$ induces a central isogeny from
$G_B$ to $G_{A_1}$.
\end{proof}
So we have the central isogenies $ G_B\times G_{A_2}
\twoheadrightarrow G_{A_1}\times G_{A_2}\twoheadleftarrow G_A
 $ of groups over $\Z_p$.
Put $\sG_{A_i}=G_{A_i}\times_{\Z_p}W(k)$. Taking the base change
to $W(k)$ one obtains central isogenies of groups over $W(k)$:
$$
\sG_B\times
\sG_{A_2}\stackrel{\xi_1}{\twoheadrightarrow}\sG_{A_1}\times
\sG_{A_2}\stackrel{\xi_2}{\twoheadleftarrow} \sG_A.
$$
For a certain natural number $l$, the cocharacter $\G_m\to
\sG_{A_1}\times \sG_{A_2}$, which is the composite
$$
\G_m\stackrel{x\mapsto x^l}{\longrightarrow} \G_m
\stackrel{\mu_{B'}\times {\rm id}}{\longrightarrow} \sG_B\times
\sG_{A_2}\stackrel{\xi_1}{\longrightarrow} \sG_{A_1}\times
\sG_{A_2},
$$
lifts to a cocharacter $\nu: \G_m\to \sG_A$. By Proposition
\ref{structure of D_crys in lattice level}, the reduction of $\nu$
modulo $p$ induces the same filtration as given by $Fil^1_A\otimes
k$ on $M_A\otimes k$. Then the filtration on $M_A$ defined by
$\nu$ corresponds to a $p$-divisible group $A'$ over $W(k)$
lifting the $p$-divisible $A\otimes k$ over $k$. We call $\nu$ in
the following by $\mu_{A'}$. One discusses the cocharacter
$\xi_1\circ \mu_{B'}$ similarly and obtains then a $p$-divisible
group $A_1'$ over $W(k)$ lifting $A_1\otimes k$. It follows that
one has an isomorphism of filtered $\phi$ modules similar to that
in Proposition \ref{structure of D_crys in lattice level} for the
filtered Dieudonn\'{e} module of $A$ by replacing $A_1$ with
$A_1'$ and $B$ in Proposition \ref{tensor factor has a filtered
phi module structure} with $B'$. Consider the opposite unipotents
$U_{\sG_B}\times id$ (resp. $U_{\sG_{A_1}}\times id$ and
$U_{\sG_A}$) induced by the cocharacter $\mu_{B'}\times id$ (resp.
$\xi_1\circ(\mu_{B'}\times id)$ and $\mu_{A'}$). By the
construction, $\xi_1$ (resp. $\xi_2$) restricts to an isogeny from
$U_{\sG_B}\times id$ to $U_{\sG_{A_1}}\times id$. (resp. from
$U_{\sG_A}$ to $U_{\sG_{A_1}}\times id$). Thus taking the
completion along the identity section, one obtains an isomorphism
$$
\hat{\xi}_{cris}=\hat{\xi}_{2}^{-1}\circ\hat{\xi}_{1}:
\hat{U}_{\sG_B} \stackrel{\cong}{\longrightarrow} \hat{U}_{\sG_A}.
$$
Let $(\sN_A,Fil^1_{\sN_A},\nabla_{\sN_A},\phi_{\sN_A})$ be the
following filtered Dieudonn\'{e} module over $\hat{U}_{\sG_A}$:
let $R_{\sG_A}$ be the complete local ring of $\hat{U}_{\sG_A}$
and $\phi_{\hat{U}_{\sG_A}}: R_{\sG_A}\to R_{\sG_A}$ the lifting
of the absolute Frobenius obtained by pulling back the
$\phi_{\hat{U}_{\sG_B}}$ via $\hat{\xi}_{cris}^{-1}$. The triple
$$
(\sN_A=M_A'\otimes_{W(k)}R_{\sG_A},
Fil^1_{\sN_A}=Fil^1_{M_A'}\otimes_{W(k)}R_{\sG_A},
\phi_{\sN_A}=u\circ(\phi_{M_A'}\otimes \phi_{\hat{U}_{\sG_A}})),
$$
where $u$ is the tautological $R_{\sG_A}$-point of $U_{\sG_A}$,
together with the connection $\nabla_{\sN_A}$ deduced from Theorem
10 \cite{Fa3}, makes the quadruple
$(\sN_A,Fil^1_{\sN_A},\nabla_{\sN_A},\phi_{\sN_A})$ an object in
$\mathcal{MF}_{[0,1]}^{\nabla}(R_{\sG_A})$. We denote again by
$\hat{\xi}_{cris}$ the equivalence of categories from
$\mathcal{MF}_{[0,1]}^{\nabla}(\hat U_{\sG_B})$ to
$\mathcal{MF}_{[0,1]}^{\nabla}(\hat U_{\sG_A})$ induced by the
isomorphism $\hat{\xi}_{cris}$.
\begin{theorem}\label{theorem on tensor decomposition over local base}
One has a natural isomorphism in the category
$\mathcal{MF}_{[0,1]}^{\nabla}(\hat{M}_{x_0})$:
$$
(H,F,\phi,\nabla)|_{\hat{M}_{x_0}}\cong\{\hat{\xi}_{cris}[\otimes_{i=0}^{r-1}(\sN_i,Fil^1_{\sN_i},\nabla_{\sN_i}),\phi_{ten}]\otimes
(M_{A_2},Fil^1_{A_2},\phi_{A_2},d)\}^{\oplus 2^{\epsilon(D)}},
$$
where
$[\otimes_{i=0}^{r-1}(\sN_i,Fil^1_{\sN_i},\nabla_{\sN_i}),\phi_{ten}]\in
\mathcal{MF}_{[0,1]}^{\nabla}(\hat U_{\sG_B})$ is the one introduced
after Proposition \ref{claim 1} and
$(M_{A_2},Fil^1_{A_2},\phi_{A_2},d)$ is a constant unit crystal with
the trivial connection.
\end{theorem}
\begin{proof}
From Proposition 2.3.5 \cite{Ki} and its proof, one knows that
$\hat{M}_{x_0}=\hat{U}_{\sG_A}$ is the deformation space of the
$p$-divisible group $A_0$ with Tate cycles $\subset M_A^{\otimes}$
fixed by the group $\sG_{A}\subset \GL_{W(k)}(M_A)$. By the remarks
of Faltings \S7 \cite{Fa3}, the above quadruple
$(\sN_A,Fil^1_{\sN_A},\nabla_{\sN_A},\phi_{\sN_A})$ gives an
explicit description in the category
$\mathcal{MF}^{\nabla}_{[0,1]}(\hat{M}_{x_0})$ of the restriction
$(H,F,\phi,\nabla)|_{\hat{M}_{x_0}}$. The decomposition of the
triple $(\sN_A,Fil^1_{\sN_A},\phi_{\sN_A})$ follows from the above
description of the universal filtered Dieudonn\'{e} module and the
corresponding statement of Proposition \ref{structure of D_crys in
lattice level}. Then the connection decomposes accordingly: we equip
the decomposition with the connection
$\nabla_{dec}:=(\otimes_{i=0}^{r-1} \nabla_{\sN_i}\otimes d)^{\oplus
2^{\epsilon(D)}}$. Then the decomposition of $\phi_{\sN_A}$ shows
that it is horizontal with respect to both $\nabla_{dec}$ and
$\nabla_{\sN_A}$. By the uniqueness of such a connection (see proof
of Theorem 10 \cite{Fa3}), $\nabla_{\sN_A}$ is isomorphic to
$\nabla_{dec}$ as claimed.
\end{proof}
The following consequence of the previous result will be used in the
next section.
\begin{corollary}\label{claim 3}
One has an isomorphism in the category
$\mathcal{MF}_{big,r}^{\nabla}(\hat{M}_{x_0})$:
$$
(H,F,\phi^r,\nabla)|_{\hat{M}_{x_0}}\cong\{\hat{\xi}_{cris}\otimes_{i=0}^{r-1}(\sN_i,Fil^1_{\sN_i},\phi_{\sN_i},\nabla_{\sN_i})\otimes
(M_{A_2},Fil^1_{A_2},\phi^r_{A_2},d)\}^{\oplus 2^{\epsilon(D)}}.
$$
\end{corollary}

\section{Second tensor power of the universal filtered Dieudonn\'{e}e module and a mass formula}\label{Section on weak Hasse-Witt pair}
Let $f_0: X_0\to M_0$ be the reduction of the universal abelian
scheme modulo $\mathfrak p$. In this section we construct a pair
$(\sP_0,\tilde F_{rel})$ over $M_0\otimes \bar k$, where $\sP_0$ is
a line bundle of negative degree and $\tilde F_{rel}:
F_{M_0}^{*r}\sP_0\to \sP_0$ is a nonzero morphism. We show that the
reduced zero divisor of $\tilde F_{rel}$ is equal to the
supersingular locus, and the multiplicity at each supersingular
point is two.

\subsection{Preliminary discussion}\label{preliminary discussion}
In this paragraph, we collect Faltings's results
(\cite{Fa2},\cite{Fa3}) into a form which we can apply in the
following conveniently. Note also that $M$ in the following
discussion could be relaxed to be an arbitrary smooth proper scheme
over $W(k)$. Let $U=\Spec R\subset M$ be a small affine subset,
which means that there is an \'{e}tale map $W(k)[T^{\pm}]\to R$. Let
$\bar R$ be the maximal extension of $R$ which is \'{e}tale in
characteristic zero (see Ch. II a) \cite{Fa2}) and
$\Gamma_R=\Gal(\bar R|R)$ be the Galois group. Let
$\mathcal{MF}_{[0,p-2]}^{\nabla}(R)$ be the category introduced in
\S3 \cite{Fa3}, and $\Rep_{\Z_p}(\Gamma_R)$ the category of
continuous representations of $\Gamma_R$ on free $\Z_p$-modules of
finite rank. By the fundamental theorem (Theorem 5* \cite{Fa3}),
there is a fully faithful contravariant functor
$$\DD: \mathcal{MF}_{[0,p-2]}^{\nabla}(R)\to \Rep_{\Z_p}(\Gamma_R).$$ An
object lying in the image of the functor $\DD$ is called a
\emph{dual crystalline representation}\footnote{It is said to be
dual because the functor $\DD$ maps the first crystalline cohomology
of an abelian variety to the dual of the first \'{e}tale cohomology.
See Theorem 7 \cite{Fa3}.}. For our convenience, we shall also
consider the covariant functor $\DD^\vee$, which maps an object $H\in
\mathcal{MF}_{[0,p-2]}^{\nabla}(R)$ to the dual of $\DD(H)$ in
$\Rep_{\Z_p}(\Gamma_R)$, and call an object in the image of $\DD^\vee$
a crystalline representation. The $p$-torsion analogue of the above
theorem is established in \cite{Fa2}. For clarity of exposition, we
use the subscript tor to distinguish the torsion analogues. So there
is also a fully faithful functor (Theorem 2.6 loc. cit.)
$$
\DD_{tor}: \mathcal{MF}_{[0,p-2]}^{\nabla}(R)_{tor}\to
\Rep_{\Z_p}(\Gamma_R)_{tor}.
$$
It follows from the construction that for an object $H\in
 \mathcal{MF}_{[0,p-2]}^{\nabla}(R)$, one has
$\DD(H)=\lim_{\infty \leftarrow n}\DD_{tor}(\frac{H}{p^nH})$.
Faltings has defined an adjoint functor $\EE_{tor}$ of $\DD_{tor}$
(see Ch. II. f)-g) \cite{Fa2}). For an object $\L\in
\Rep_{\Z_p}(\Gamma_R)$, one defines
$$
\EE(\L):=[\lim_{\infty \leftarrow
n}\EE_{tor}(\frac{\L}{p^n\L})]/\textrm{torsion}.
$$
Clearly, for $\L=\DD(H)$, it holds that
$$
\EE(\L)=\lim_{\infty \leftarrow
n}\EE_{tor}(\frac{\L}{p^n\L})=\lim_{\infty \leftarrow
n}\frac{H}{p^nH}=H.$$ Finally define $\EE^\vee(\L):=\EE(\L^\vee)$.
\begin{lemma}\label{direct summan of dual crystalline sheaf}
Suppose $\W,\W_1,\W_2\in \Rep_{\Z_p}(\Gamma_R)$. The following
basic properties hold:
\begin{itemize}
    \item [(i)] Suppose $\W=\W_1\oplus \W_2$. Then $\W$ is crystalline if and only if
each $\W_i$ is so.
    \item [(ii)] Suppose $\W$ crystalline with Hodge-Tate weight $n$ and a Schur functor $\SSS_{\lambda}$ with $\lambda$ a partition of $m\leq
p-1$ satisfying $mn\leq p-2$. Then $\SSS_{\lambda}(\W)$ is still crystalline, and there is
a natural isomorphism $\EE^\vee(\SSS_{\lambda}\W)\cong\SSS_{\lambda}
\EE^\vee(\W)$.
    \item [(iii)] Suppose $\W_i, i=1,2$ crystalline with Hodge-Tate weights $n_i$ satisfying $n_1n_2\leq p-2$. Then $\W_1\otimes \W_2$ is crystalline, and there is a natural
isomorphism
$$\EE^\vee(\W_1\otimes \W_2)\cong\EE^\vee(\W_1)\otimes \EE^\vee(\W_2).$$
\end{itemize}
\end{lemma}
\begin{proof}
Consider $\EE_{tor}(\W/p^n)=\EE_{tor}(\W_1/p^n)\oplus
\EE_{tor}(\W_2/p^n)$. By Ch. II g) \cite{Fa2}, one has always that
$l(\EE_{tor}(\W_i/p^n))\leq l(\W_i/p^n)$, and the equality holds
iff $\W_i/p^n$ lies in the image of $\DD_{tor}$. Now assume $\W$
to be dual crystalline, that is $\W=\DD(H)$. So
$\frac{\W}{p^n\W}=\frac{\DD(H)}{\DD(p^nH)}=\DD_{tor}(\frac{H}{p^nH})$.
Hence from
$$
l(\W/p^n)=\sum_i l(\W_i/p^n)\geq
\sum_il(\EE_{tor}(\W_i/p^n))=l(\EE_{tor}(\W/p^n)),
$$
it follows that there are $H_{i,n}\in
\mathcal{MF}_{[0,p-2]}^{\nabla}(R)_{tor}, i=1,2$ such that
$\DD_{tor}(H_{i,n})=\W_i/p^n$ and by the faithfulness of
$\DD_{tor}$, $H_{1,n}\oplus H_{2,n}=H/p^n$. Taking the inverse
limit, one obtains $H_i=\lim_{\infty \leftarrow n}H_{i,n}$ with the
equality $H_1\oplus H_2=H$, which implies that $H_i$ is torsion free
and is an object in $\mathcal{MF}_{[0,p-2]}^{\nabla}(R)$. Thus it
follows that
$$
\DD(H_i)=\lim_{\infty \leftarrow n}\DD_{tor}(H_i/p^n)=\lim_{\infty
\leftarrow n}\W_i/p^n=\W_i,
$$
and thereby $\W_i$ is dual crystalline. The other direction of (i)
is obvious. Clearly (ii) follows from (iii). To show (iii), it is to
show that for $H_i\in \mathcal{MF}_{[0,p-2]}^{\nabla}(R),i=1,2$,
there is a natural isomorphism $ \DD(H_1)\otimes
\DD(H_2)\cong\DD(H_1\otimes H_2)
 $.
Taking an element $f_i\in \DD(H_i)$, which is an $\hat R$-linear
map from $H_i$ to $B^+(R)$ respecting the filtrations and the
$\phi$s, one forms the $\hat R$-linear map $f_1\otimes f_2:
H_1\otimes H_2\to B^+(R)$. It respects the filtrations and the
$\phi$s and therefore gives an element in $\DD(H_1\otimes H_2)$.
So one has a natural map $\DD(H_1)\otimes \DD(H_2)\to
\DD(H_1\otimes H_2)$, which is obviously injective. Because both
sides have the same $\Z_p$-rank, it remains to show that the
quotient $\DD(H_1\otimes H_2)/\DD(H_1)\otimes\DD(H_2)$ has no
torsion. For that we pass to modulo $p$ reduction and use the
functor $\DD_{tor}$. The same argument as above applied to $H_i/p$
shows that the $\F_p$-linear map $\DD_{tor}(H_1/p)\otimes
\DD_{tor}(H_2/p)\to \DD_{tor}(H_1\otimes H_2/p)$ is injective and
therefore is bijective. This shows the non $p$-torsioness.
\end{proof}
Let $\sU=\{U\}$ be a small affine open covering of $M$. Theorem 2.3
\cite{Fa2} shows that one can define the global category
$\mathcal{MF}^{\nabla}_{[0,p-2]}(M)$. Furthermore Faltings explained
that these various local functors $\DD_{tor}$ glue to a global one
from $\mathcal{MF}^{\nabla}_{[0,p-2]}(M)_{tor}$ to
$\Rep_{\Z_p}(\pi_1(M^0))_{tor}$ (see page 42 \cite{Fa2}). By passing
to limit, one obtains a global functor $\DD:
\mathcal{MF}^{\nabla}_{[0,p-2]}(M)\to \Rep_{\Z_p}(\pi_1(M^0))$. An
object in the image of $\DD$ is called a \emph{dual crystalline
sheaf}. Similarly, one defines $\DD^\vee$ and $\EE^\vee$ in the global
setting and calls an object in the image of $\DD^\vee$ a crystalline
sheaf. Now let $\W$ be a crystalline sheaf of $M^0$ and $H$ the
corresponding filtered Frobenius crystal to $\W$ (i.e.
$\DD^\vee(H)=\W$). Let $x$ be a $W(k)$-valued point of $M$. Consider
the specialization of both objects into the point $x$: Via the
splitting of the short exact sequence
$$
1\to \pi_1(\bar M^0) \to \pi_1(M^0)\to \Gal_{{\rm Frac}(W(k))}\to
1
$$
induced by the point $x^0: {\rm Frac}(W(k))\to M^0$, $\W_{x^0}$ is a
representation of $\Gal_{{\rm Frac}(W(k))}$. On the other hand,
$H_x$ is obviously an object in $\mathcal{MF}_{[0,p-2]}(W(k))$.
\begin{lemma}\label{global to local}
Notation as above. Then the following statements hold:
\begin{itemize}
    \item [(i)] The Galois representation $\W_{x^0}\otimes
\Q_p$ is crystalline in the sense of Fontaine.
    \item [(ii)] $H_{x}$ is naturally a strong divisible
lattice of the ${\rm Frac}(W(k))$-vector space
$D_{crys}(\W_{x^0}\otimes \Q_p)$ in the sense of
Fontaine-Laffaille (\cite{FL}).
    \item [(iii)] There is a natural isomorphism of $\Z_p[\Gal_{{\rm
Frac}(W(k))}]$-modules:
$$
\DD^\vee(H_{x})\cong\DD^\vee(H)_{x^0}.
$$
Consequently, there is a natural isomorphism in
$\mathcal{MF}_{[0,p-2]}(W(k))$:
$$
\EE^\vee(\W)_{x}\cong\EE^\vee(\W_{x^0}).
$$
\end{itemize}
\end{lemma}
\begin{proof}
It is clear that we can pass the problem to a small affine subset
$U=\Spec R\subset M$. Choose a local coordinate $T$ of $R$ such
that the $W(k)$-point $x$ of $R$ is given by $T=1$ (i.e. the
composite $W(k)[T^{\pm}]\to R\to W(k)$ is the $W(k)$-morphism
determined by $T\mapsto 1$). Fix an isomorphism $\bar
R\otimes_RW(k)\cong\overline{W(k)}$ such that the following
diagram commutes:
$$
\xymatrix{
  \Spec \overline{W(k)} \ar[d]_{} \ar[r]^{\bar x} & \Spec \bar R \ar[d]^{} \\
  \Spec W(k) \ar[r]^{x} & \Spec R.   }
$$
Note that the subgroup $\Gamma_{R,x}\subset \Gamma_R$ preserving
the prime ideal $\ker(\bar R\to \overline{W(k)})$ of $\bar x$ is
naturally isomorphic to $\Gal_{{\rm Frac}(W(k))}$, and it is equal
to the image of the splitting $\Gamma_R\twoheadrightarrow
\Gal_{{\rm Frac}(W(k))}$ induced by the $W(k)$-point $x$. Fix a
Frobenius lifting $\phi$ of $\hat R$ which fixes the point $x$
(e.g the one determined by $T\mapsto T^p$). Note also that the
point $\bar x: \bar R\to \overline{W(k)}$ induces a surjection of
$B^+(W(k))$-algebras $B^+(\hat R)\to B^+(W(k))$, which preserves
the filtration and the Frobenius. Recall that
$$
\DD(H)=\Hom_{\hat R,Fil,\phi}(H,B^+(\hat R))=(H^\vee\otimes_{\hat
R}B^+(\hat R))^{Fil=0,\phi=1},
$$ and
$$
\DD(H_{x})=\Hom_{W(k),Fil,\phi}(H_x,B^+(W(k)))=(H_x^\vee\otimes_{W(k)}B^+(W(k)))^{Fil=0,\phi=1}.
$$
The above free $\Z_p$-modules (say of rank $n$) are basically
obtained by solving certain equations (see page 127-128 \cite{Fa3}
and page 37-38 \cite{Fa2}). There are also natural surjections
$$
B^+(\hat R)\twoheadrightarrow B^+(\hat R)/p\cdot B^+(\hat
R)\twoheadrightarrow \bar R/p\cdot\bar R,
$$
and similarly for $B^+(W(k))$. These make the following diagrams
commute:
$$
\xymatrix{
  B^+(\hat R) \ar[d]_{ } \ar[r]^{ } & B^+(\hat R)/p\cdot B^+(\hat
R) \ar[d]_{ } \ar[r]^{ } & \bar R/p\cdot\bar R \ar[d]^{} \\
  B^+(W(k)) \ar[r]^{} &B^+(W(k))/p\cdot B^+(W(k))  \ar[r]^{} &
\overline{W(k)}/p\cdot\overline{W(k)}, }
$$
where the vertical arrows are induced by the point $\bar x$.
Faltings showed loc. cit. that it suffices to solve the equations
over the quotient $\bar R/p$ (resp. $\overline{ W(k)}/p$) because
each solution over the quotient can be uniquely lifted. Now choose
a filtered basis $\{h_{i}\}$ of $H$, which restricts to a filtered
basis of $H_x$. An element of $\DD(H)$ is then given by an
$n$-tuple in $B^{+}(\hat R)$ satisfying a system of equations
coming from the condition on $\phi$s. For each such an $n$-tuple,
we obtain an $n$-tuple in $B^+(W(k))$ by projecting each component
to $B^+(W(k))$ (the projection $B^+(\hat R)\twoheadrightarrow
B^+(W(k))$ induced by the point $\bar x$). As the projection
preserves the filtration and the Frobenius, and as the filtration
and the Frobenius on $H_x$ are the one of $H$ by restriction, any
so-obtained $n$-tuple satisfies the equations required for
$\DD(H_x)$. So we have a $\Z_p$-linear map
$$
ev_x: \DD(H)\to \DD(H_x), f\mapsto f(x).
$$
Consider first the Galois action. Recall that $\Gal_{{\rm
Frac}(W(k))}$ acts on $\DD(H_x)\subset H_x^\vee\otimes_{W(k)}B^+(W(k))$
on the second tensor factor. But the $\Gamma_R$-action on
$\DD(H)\subset H^\vee\otimes_{\hat R}B^+(\hat R)$ must be also
intertwined with the connection $\nabla$ on the first factor.
However the restriction to the subgroup $\Gamma_{R,x}$ does not
involve the connection (see Ch. II e) \cite{Fa2} for the $p$-torsion
situation which we can also assume in the argument). So the above
map $ev_x$ is equivariant with respect to $\Gamma_{R,x}$-action on
$\DD(H)$ and $\Gal_{{\rm Frac(W(k))}}$ action on $\DD(H_x)$.\\
Next we claim that $ev_x$ is a $\Z_p$-isomorphism. For that we
consider the base change $ev_x\otimes \Q_p$ and then the reduction
$ev_x\otimes \F_p$. By Ch. II h) \cite{Fa2}, one has a natural
isomorphism
$$
H\otimes_{\hat R}B(\hat R)\cong\DD^\vee(H)\otimes_{\Z_p}B(\hat R),
$$
which respects the $\Gamma_R$-actions, filtrations and $\phi$s.
Tensorizing the above isomorphism with $B(W(k))$ as $B(\hat
R)$-modules (the morphism $B(\hat R)\twoheadrightarrow B(W(k))$
induced by $\bar x$) and taking the $\Gamma_{R,x}$-invariance of
both sides, we obtain an isomorphism of $\Gal_{{\rm
Frac}(W(k))}$-representations:
$$
V_{crys}(H_x\otimes {\rm Frac}(W(k)))\cong\DD^\vee(H)_{x^0}\otimes
\Q_p.
$$
That is, there is a natural isomorphism $ \DD(H)_{x^0}\otimes
\Q_p\cong V_{crys}^\vee(H_x\otimes {\rm Frac}(W(k)))
 $.
By Fontaine-Laffaille (see \S7-8 in \cite{FL}, see also \S2
\cite{Br}), $\DD(H_x)$ is a Galois lattice of
$V_{crys}^\vee(H_x\otimes {\rm Frac}(W(k)))$ by the isomorphism, and
$H_x$ is a strong divisible lattice of
$D_{crys}(\DD^\vee(H)_{x^0}\otimes \Q_p)$. This shows (i) and (ii).
Also it implies that the map $ev_x\otimes \Q_p$ is an isomorphism.
In particular the map $ev_x$ is injective. Using the fact that the
composite $B^+(W(k))\to B^+(\hat R)\to B^+(W(k))$ is the identity,
one sees that $ev_x(\DD(H))\cap p\DD(H_x)=ev_x(p\DD(H))$, and the
map $ev_x\otimes \F_p: \DD(H)/p\DD(H)\to \DD(H_x)/p\DD(H_x)$ is
therefore injective. Now that the $\F_p$-vector spaces
$\DD(H)/p\DD(H)$ and $\DD(H_x)/p\DD(H_x)$ have the same dimension
$n$, $ev_x\otimes \F_p$ is an isomorphism. Thus $ev_x$ is an
isomorphism. This proves (iii).
\end{proof}
Let $r\in \N$ be a natural number. Let
$\Rep_{\Z_{p^r}}(\pi_1(M^0))\subset \Rep_{\Z_p}(\pi_1(M^0))$ be the
full subcategory of $\Z_{p^r}[\pi_1(M^0)]$-modules. An object which
lies in both $\Rep_{\Z_{p^r}}(\pi_1(M^0))$ and the image of $\DD^\vee$
is called a $\Z_{p^r}$-crystalline sheaf. One notes that the proof
of Theorem 2.3 \cite{Fa2} works verbatim to show that the local
categories $\{\mathcal{MF}_{big,r}^{\nabla}(U)\}_{U\in \sU}$ (see
\S\ref{Drinfel'd module and versal deformation subsection}) glue
into a global category $\mathcal{MF}_{big,r}^{\nabla}(M)$. A typical
object in this category is obtained by replacing the Frobenius of an
object in $\mathcal{MF}_{[0,p-2]}^{\nabla}(M)$ with its $r$-th
power.
\begin{lemma}\label{eigendecomposition of families of filtered
crystals} Let $\W$ be a $\Z_{p^r}$-crystalline sheaf. Assume that
$\Z_{p^r}\subset \sO_M$. Then there is a natural decomposition $
 \EE^\vee(\W)=\oplus^{r-1}_{i=0}\EE^\vee(\W)_i
 $ in
the category $\mathcal{MF}_{big, r}^{\nabla}(M)$.
\end{lemma}
\begin{proof}
The multiplication by $s\in \Z_{p^r}$ on $\W$ commutes with
$\pi_1(M^0)$-action. Hence it gives rise to an endomorphism
$s_{\mathcal{MF}}$ of $\EE^\vee(\W)$ in the category
$\mathcal{MF}_{[0,p-2]}^{\nabla}(M)$. By assumption $\sO_M$ contains
the eigenvalues of $s_{\mathcal{MF}}$. The eigen decomposition of
$\EE^\vee(\W)$ with respect to $s_{\mathcal{MF}}$ gives rise to a
decomposition of form $\oplus^{r-1}_{i=0}\EE^\vee(\W)_i$ such that the
direct factors are preserved by $\nabla$ and permutes cyclically by
$\phi$. Hence the lemma follows.
\end{proof}
\subsection{Second wedge/symmetric power of the universal filtered Dieudonn\'{e} module}
From now on the rational prime number $p$ is assumed to be $\geq
5$ in addition to Assumption \ref{assumption on p from integral
model}. The aim of the paragraph is to show a direct sum
decomposition of the second wedge (resp. symmetric) of the
universal filtered Dieudonn\'{e} crystal for $d$ an even (resp.
odd) number. Recall that by Corollary \ref{tensor decompositon of
local system} we have a direct-tensor decomposition of \'{e}tale
local systems $ \HH=(\V\otimes \U)^{\oplus 2^{\epsilon(D)}}
 $.
It follows from Lemma \ref{direct summan of dual crystalline
sheaf} that the direct summand $\HH':=\V\otimes \U$ of $\HH$ is
crystalline. As $\det(\HH)\cong\Z_p(-2^{d+\epsilon(D)-1})$, it
follows that
$$
\det\HH'\cong\Z_p(-2^{d-1})\otimes \chi,
$$
where $\chi$ is a 2-torsion crystalline sheaf (which is trivial when $\epsilon(D)=0$). Consider the following $\Z_{p^d}$-\'{e}tale local system:
$$
\tilde{\HH}':=(\V\otimes\det(\V)^{-\frac{1}{2}}\otimes
\U\otimes\det(\U)^{-\frac{1}{2}})\otimes_{\Z_p}\Z_{p^d}.
$$
Because of the equality
$$
\bigwedge^2(\tilde{\HH}')=[\bigwedge^2(\HH')\otimes\det(\HH')^{-1}]\otimes_{\Z_p} \Z_{p^d},
$$
$\bigwedge^2(\tilde{\HH}')$ is a $\Z_{p^d}$-crystalline sheaf. For $1\leq i\leq r$, put
$$
\tilde{\V}_i=\V_{1,\sigma^{i-1}}\otimes
\det(\V_1)^{-\frac{1}{2}},\
\tilde{\V}'_i=\tilde{\V}_i\otimes_{\Z_{p^r}}\Z_{p^d},
$$
and for $r+1=r_1+1\leq i\leq r_1+r_2$, put
$$
\tilde{\V}'_i=(\U_{1,\sigma^{i-1}}\otimes
\det(\U_1)^{-\frac{1}{2}})\otimes_{\Z_{p^{r_2}}}\Z_{p^d},
$$
and so on. Then by Corollary \ref{tensor decompositon of local system},
we have a decomposition of $\tilde{\HH}'$ into tensor product of
rank two $\Z_{p^d}$-\'{e}tale local systems:
 $
\tilde{\HH}'=\otimes_{i=1}^{d}\tilde{\V}'_i
 $.
In the tensor decomposition, we assume that the factor
$\tilde{\V}_1$ corresponds to the place $\tau$ (see Lemma
\ref{tensor decomposition of corestriction algebra over local
fields}).
\begin{remark}\label{each tesnor factor is dual crystalline sheaf}
We conjecture that each tensor factor $\tilde{\V}'_i$ in the above
decomposition is a $\Z_{p^d}$-crystalline sheaf. The next lemma
shows that $\Sym^2\tilde{\V}'_i$ is a direct factor of a
crystalline sheaf and therefore crystalline by Lemma \ref{direct
summan of dual crystalline sheaf} (i).
\end{remark}
The following lemma is proved by induction on $d$:
\begin{lemma}\label{schur functor}
For $I=(i_1,\cdots,i_l)$ a multi-index in $\{1,\cdots,d\}$, put
 $
\Sym^2(\tilde{\V}')_I:=\otimes_{j=1}^{l}\Sym^2\tilde{\V}'_{i_j}
 $.
One has a direct sum decomposition of $\Z_{p^d}$-\'{e}tale local
systems:
\begin{itemize}
    \item [(i)] for $d$ even,
    $$
    \bigwedge^2(\tilde{\HH}')=\bigoplus_{I,|I|\
    \rm{odd}}\Sym^2(\tilde{\V}')_I,\quad  \Sym^2(\tilde{\HH}')=\bigoplus_{I,|I|\
    \rm{even}}\Sym^2(\tilde{\V}')_I,
    $$
    \item [(ii)] for $d$ odd,
    $$
    \bigwedge^2(\tilde{\HH}')=\bigoplus_{I,|I|\
    \rm{even}}\Sym^2(\tilde{\V}')_I,\quad  \Sym^2(\tilde{\HH}')=\bigoplus_{I,|I|\
    \rm{odd}}\Sym^2(\tilde{\V}')_I.
    $$
\end{itemize}
\end{lemma}
In the following we shall focus on the direct summand
$\Sym^2(\tilde{\V}_1)$ in the decomposition since it is so to
speak the (rank three) uniformizing direct factor of the weight
two integral $p$-adic variation of Hodge structures of the
universal family. Also one notices that this factor is actually
defined over $\Z_{p^r}$. So by taking the
$\Gal(\Z_{p^d}|\Z_{p^r})$-invariants, one obtains a direct
decomposition into $\Z_{p^r}$-dual crystalline sheaves with
$\Sym^2\tilde{\V}_1$ as a direct factor for the second wedge
(resp. symmetric) power for $n$ even (resp. odd).
\begin{proposition}\label{direct sum decomposition of second/wedge
power} Let $(H',F,\phi,\nabla)\in \mathcal{MF}^{\nabla}_{[0,1]}(M)$
be the sub filtered $F$-crystal corresponding to the factor
$\HH'\subseteq \HH$. One has a direct sum decomposition in
$\mathcal{MF}^{\nabla}_{big,r}(M)$:
\begin{itemize}
    \item [(i)] for $d$ even,
    $$
    \bigwedge^2(H',F,\phi^r,\nabla)=\bigoplus_{i=1}^{r}\EE^\vee(\Sym^2\tilde{\V}_i)_{0}\{
    -2^{d-1}\}\otimes \EE^{\vee}(\chi) \oplus \textrm{rest term}.
$$
    \item [(ii)] for $d$ odd,
    $$
    \Sym^2(H',F,\phi^r,\nabla)=
\bigoplus_{i=1}^{r}\EE^\vee(\Sym^2\tilde{\V}_i)_{0}\{
    -2^{d-1}\}\otimes  \EE^{\vee}(\chi) \oplus \textrm{rest term}.
$$
\end{itemize}
\end{proposition}
\begin{proof}
We shall prove (i) only because (ii) can be similarly proved. By the
discussion before the proposition, we obtain a decomposition in
$\mathcal{MF}^{\nabla}_{[0,2]}(M)$:
$$
\EE^\vee(\bigwedge^2(\HH')\otimes\chi^{-1}\otimes\Z_{p^r})=
\bigoplus_{i=1}^{r}\EE^\vee(\Sym^2\tilde{\V}_i)\{-2^{d-1}\} \oplus \textrm{rest
term}.
$$
The claimed decomposition is obtained by considering the eigen
decomposition of both sides corresponding the eigenvalue $s_0$:
The argument is similar to that of Lemma \ref{lattic structure of
fontaine functors}. The right hand side is clear, and the question
is the left hand side. It suffices to consider the eigen component
after inverting $p$. By Ch. II h) \cite{Fa2}, one has a
$\Gamma_R$-isomorphism
$$
\EE^\vee(\bigwedge^2(\HH')\otimes\chi^{-1}\otimes\Z_{p^r}|_{\hat
U})\otimes_{\hat R}B(R)\cong\bigwedge^2(\HH')\otimes \chi^{-1}\otimes
\Z_{p^r}|_{\hat U})\otimes_{\Z_p}B(R).
$$
It follows that
\begin{eqnarray*}
   [\EE^\vee(\bigwedge^2(\HH')\otimes \chi^{-1}\otimes
\Z_{p^r})|_{\hat U}]_0[\frac{1}{p}]&\cong&
[\bigwedge^2(\HH')\otimes_{\Z_p}\chi^{-1}\otimes_{\Z_p}
\Z_{p^r}|_{\hat U}\otimes_{\Z_{p^r}}B(R)]^{\Gamma_R} \\
   &\cong& [\bigwedge^2\HH'\otimes_{\Z_p} \chi^{-1}\otimes_{\Z_p}
 B(R)]^{\Gamma_R} \\
   &\cong& \EE^\vee(\bigwedge^2\HH')\otimes \EE^{\vee}(\chi^{-1})[\frac{1}{p}]\\
   &\cong&  \bigwedge^2(\EE^\vee(\HH'))\otimes \EE^{\vee}(\chi^{-1})[\frac{1}{p}].
\end{eqnarray*}
This shows that the eigen submodule of
$\EE^\vee(\bigwedge^2\HH'\otimes\chi^{-1}\otimes\Z_{p^r})$ to the
eigenvalue $s_0$ is naturally isomorphic to $\bigwedge^2H'\otimes \EE^{\vee}(\chi^{-1})$. The claimed decomposition then follows.
\end{proof}
\subsection{Construction of the pair}\label{subsection on weak Hasse-Witt
pair} The aim of the paragraph is to construct the pair
$(\sP_0,\tilde F_{rel})$ claimed in the introduction of the section.
In the following $\EE_0$ denotes for
$\EE^\vee(\Sym^2\tilde{\V}_1)_{0}\{-2^{d-1}\}$ and $\tilde{\EE}_0=\EE_0\otimes  \EE^{\vee}(\chi)$ for the factor in the decomposition in Proposition \ref{direct sum decomposition of second/wedge power}. Note that the square of $\EE^{\vee}(\chi)$ is the trivial crystal. In particular, its filtration is trivial and its restriction to each $W(k)$-valued point is a unit crystal. Let $x_0\in M_0(k)$ be a $k$-rational point and $x$ a $W(k)$-valued point of $M$ lifting $x_0$.
\begin{proposition}\label{application of global to local}
One has a natural isomorphism in the category
$\mathcal{MF}^{\nabla}_{big, r}(\hat M_{x_0})$:
$$
\EE_0|_{\hat M_{x_0}}\cong\hat \xi_{crys}[\Sym^2\sN_0\otimes
\bigotimes_{i=1}^{r-1}\det(\sN_i)].
$$
\end{proposition}
\begin{proof}
Assume $d$ to be even. By Corollary \ref{claim 3}, one has a natural
isomorphism in $\mathcal{MF}^{\nabla}_{big, r}(\hat M_{x_0})$:
$$
H'|_{\hat M_{x_0}}\cong\hat
\xi_{crys}(\otimes_{i=0}^{r-1}\sN_i)\otimes M_{A_2}.
$$
By a Schur functor calculation as in Lemma \ref{schur functor},
one finds that via the isomorphism $\hat
\xi_{crys}[\Sym^2\sN_0\otimes \bigotimes_{i=1}^{r-1}\det(\sN_i)]$
is a direct factor of $\bigwedge^2(H'|_{\hat M_{x_0}})\otimes \EE^{\vee}(\chi^{-1})|_{\hat M_{x_0}}$. The point
is to show that it is the direct factor $\EE_0|_{\hat M_{x_0}}$.
Note that $\hat\xi_{crys}[ \Sym^2\sN_0\otimes
\bigotimes_{i=1}^{r-1}\det(\sN_i)]$ is the unique rank three
direct factor with the nontrivial filtration. So it suffices to
show that the rank three direct factor $\EE_0|_{\hat M_{x_0}}$ has
also this property. To that we show that the filtration of the
filtered $\phi^r$-module $(\EE_0)_x\otimes {\rm Frac}W(k)$ is
nontrivial. By Lemma \ref{schur functor},
$\Sym^2(\tilde{\V}_1(-2^{d-2}))$ is a direct factor of the
crystalline sheaf $\bigwedge^2(\HH'\otimes \Z_{p^r})$. So by Lemma
\ref{global to local} (i), $(\Sym^2\tilde{\V}_1(-2^{d-2}))_{x^0}$
is a crystalline lattice for the group $\Gal_{{\rm Frac}W(k)}$ and
by (iii), one has the equality (after taking the eigen component
to the eigenvalue $s_0$)
$$
(\EE_0)_x=\EE^\vee(\Sym^2(\tilde{\V}_1(-2^{d-2}))_{x^0})_0.
$$
Then by Lemma \ref{global to local} (ii) it is to determine the
filtration of $D_{crys}(\Sym^2(\tilde{\V}_1(-2^{d-2})_{x^0}\otimes
\Q_p))_0$. Consider the $\Gal_{{\rm Frac}W(k)}$-representation
$\Sym^2(\tilde{\V}_1(-2^{d-2})_{x^0}\otimes \Q_p)$. It is equal to
$\Sym^2(\tilde{\V}_{1,x^0}(-2^{d-2})\otimes \Q_p)$, and by
Proposition \ref{second crystalline tensor decomposition}
$\V_{1,x^0}\otimes \Q_p$ is crystalline for an open subgroup
$\Gal_E\subset \Gal_{{\rm Frac}W(k)}$. As
$$
\Sym^2(\tilde{\V}_{1,x^0}(-2^{d-2}))=
\Sym^2(\V_{1,x^0})\otimes_{\Z_{p^r}}\det(\V_{1,\sigma,x^0})\otimes_{\Z_{p^r}}
\cdots\otimes_{\Z_{p^r}}\det(\V_{1,\sigma^{r-1},x^0}),
$$
and the functor $D_{crys}$ commutes with a Schur functor for a
crystalline representation, we have
$$
D_{crys}(\Sym^2(\tilde{\V}_{1,x^0}(-2^{d-2})\otimes
\Q_p))_0=\Sym^2(D_{crys}(\tilde{\V}_{1,x^0}(-2^{d-2})\otimes
\Q_p)_0),
$$
which is seen to be naturally isomorphic to
 $
[\Sym^2M_0\otimes \bigotimes_{i=1}^{r-1}\det(M_i)]\otimes
{\rm Frac}(W(k_E))
 $.
This shows that the filtration of $(\EE_0)_x\otimes {\rm
Frac}W(k_E)$ is nontrivial. So is the filtration on
$(\EE_0)_x\otimes {\rm Frac}W(k)$.
\end{proof}

{\itshape Construction of $\sP_0$.} Consider the filtration on the
factor $\EE_0$. As the Hodge filtration on $H$ is filtered free (see
\S2 \cite{Fa3}), the induced filtration on $\EE_0$ by Proposition
\ref{direct sum decomposition of second/wedge power} is also
filtered free.
\begin{lemma}
The filtration $F$ on $\EE_0$ is nontrivial with form
$$
\EE_0=F^0\EE_0\supset F^1\EE_0\supset F^2\EE_0
$$
and each grading is locally free of rank one.
\end{lemma}
\begin{proof}
As it is filtered free, it suffices to show this over a point $x$ as
above. Then it follows from Proposition \ref{application of global
to local} and the proofs of Propositions \ref{second crystalline
tensor decomposition}, \ref{possible newton slopes}.
\end{proof}
By the lemma we put $\sP=\frac{\EE_0}{F^1\EE_0}=\frac{\tilde{\EE}_0}{F^1\tilde{\EE}_0}$ and $\sP_0$ be the
modulo $p$ reduction of $\sP$ which is defined over $M_0$. Next we
consider the line bundle $\sP^0$ over $M^0$ by taking a comparison
with the variation of Hodge structures at infinity associated to the
Mumford family. Let $(H_{dR}^1,F_{hod},\nabla^{GM})$ be the
automorphic vector bundle over $M_K$ coming from the universal
family of abelian varieties over $M_K$. One has a natural
isomorphism
$$
(H,F,\nabla)\otimes_{\sO_{\mathfrak p}} F_{\mathfrak
p}\cong(H_{dR}^1,F_{hod},\nabla^{GM})\otimes_{F} F_{\mathfrak p}.
$$
We intend to show a tensor decomposition of
$(H',F,\nabla)\otimes_{\sO_{\mathfrak p,\tau}} \bar{\Q}_p$ of the
form
$$
(H',F,\nabla)\otimes_{\sO_{\mathfrak p,\tau}} \bar
\Q_p=(H_1,F_1,\nabla_1)\otimes \cdots\otimes (H_d,F_d,\nabla_d),
$$
where the first tensor factor on the right hand side is the unique
one admitting the nontrivial filtration. For this we apply the
theory of de Rham cycles as developed in \S2.2 \cite{Ki}. Let
$\{s_{\alpha,B}\}\subset H_{\Q}^{\otimes}$ be a finite set of
tensors defining the subgroup $G_{\Q}\subset \GL(H_{\Q})$. By
Corollary 2.2.2 loc. cit., it defines a set of de Rham cycles
$\{s_{\alpha,dR}\}\subset (H^1_{dR})^{\otimes}$ defined over the
reflex field $\tau(F)$, which are by definition
$\nabla^{GM}$-parallel and contained in $Fil^0$.
\begin{lemma}
The set of de Rham cycles $\{s_{\alpha,dR}\}$ induces a
direct-tensor decomposition
$$
(H_{dR}^1,F_{hod},\nabla^{GM})\otimes_{F,\tau}\bar
\Q=[\bigotimes_{i=1}^{d}(H_{dR,i}^1,F_{hod,i},\nabla_i^{GM})]^{\oplus
2^{\epsilon(D)}}
$$
such that the factor $(H_{dR,1}^1,F_{hod,1},\nabla_1^{GM})$ is the
unique one with nontrivial filtration.
\end{lemma}
\begin{proof}
Let $\pi: \tilde M_{an}:=X\times G(\A_f)/K\to M_K(\C)$ be the
natural projection of complex analytic spaces. The pull-back of
$(H_{dR}^1,\nabla^{GM})\otimes \C$ over $M_K(\C)$ via $\pi$ is
trivialized, and by the de Rham isomorphism it is isomorphic to
$(H_{\Q}\otimes_{\Q} \sO_{\tilde M_{an}},1\otimes d)$. By a similar
discussion on the direct-tensor decomposition of the
$G(\Q)$-representation $H_{\Q}\otimes_{\Q} \C$ as given in \S2.1,
the tensors $s_{\alpha,B}\otimes 1$s induce a tensor decomposition
of $\pi^*((H_{dR}^1,\nabla^{GM})\otimes \C)$. It is
$G(\Q)$-equivariant by construction, and hence descends to a
decomposition on $(H_{dR}^1,\nabla^{GM})\otimes \C$. This is the
same tensor decomposition induced by the tensors $s_{\alpha,dR}$s.
Since they are defined over $\tau(F)$, the tensor decomposition
already occurs over $\bar \Q$. We have also to check the property
about the filtration in the tensor decomposition. Note that the
Hodge filtration $\pi^*(H_{dR}^1,F_{hod})\otimes \C$ over the point
$[0\times id]$ is induced from $\mu_{h_0}: \G_m(\C)\to G_{\C}\subset
\GL(H_{\C})$. The assertion follows then from the definition of
$h_0$ in \S\ref{Section on Shimura curve of Hodge type}.
\end{proof}
Composed with the embedding $\iota: \bar \Q\hookrightarrow \bar
\Q_p$, we obtain the claimed tensor decomposition on
$(H',F,\nabla)\otimes_{\sO_{\mathfrak p,\tau}} \bar{\Q}_p$. Taking
the grading with respect to $F_{hod,i}$, one obtains the
associated Higgs bundle $(E_i,\theta_i)$ with
$(H_{dR,i}^1,F_{hod,i},\nabla_i^{GM})$. By the lemma, only
$\theta_1$ is nontrivial. In fact, it is a maximal Higgs field (see
\cite{VZ}), that is,
 $$
\theta_1: F_{hod,1}^1\stackrel{\cong}{\longrightarrow}
\frac{H^1_{dR,1}}{F_{hod,1}^1}\otimes \Omega_{M_K\otimes \bar \Q}.
 $$
Actually over each connected component of $M_K$, $\theta_1\otimes
\C$ is a morphism of locally homogenous bundles of rank one. Then
it must be an isomorphism, because it will otherwise be zero, and
together with the zero Higgs fields on the other factors
$E_i,i\geq 2$, this implies that the Kodaira-Spencer map of the
universal family is trivial, which is absurd. As both
$F_{hod,1}^1$ and $\frac{H^1_{dR,1}}{F_{hod,1}^1}$ are locally
homogenous line bundles over each connected component of $M_K$,
their isomorphism classes are determined by the corresponding
representations of $K_{\R}\otimes \C$, where $K_{\R}$ is the
stabilizer of $G(\R)$ at $0\in X$. In this way one easily shows
that they are dual to each other. By putting $\sL:=F_{hod,1}^1$,
one has then
 $$
\theta_1: \sL\cong\sL^{-1} \otimes \Omega_{M_K\otimes \bar \Q}
 .$$
By abuse of notation, we use $\sL$ again to denote the base change
of $\sL$ to $$\bar M^0:=M^0\otimes_{F_{\mathfrak p}}\bar
\Q_p=(M_K\otimes_{F,\tau}\bar \Q)\otimes \bar \Q_p
 .$$
\begin{lemma}
One has a natural isomorphism
 $
\sP^0\cong\sL^{-2}
 $ over $\bar M_0$.
\end{lemma}
\begin{proof}
In fact we show that there is a natural isomorphism
 $
(\EE_0,F)\otimes \sO_{\bar M^0}\cong\Sym^2(H_1,F_1)
 $.
We raise the defining field of $M^0$ so that it contains the
defining field of $H_i, 1\leq i\leq d$ and $\Q_{p^d}$. By abuse of
notation, we use the same notation to mean an above object after
the base change. Let $U=\Spec R\subset M$ be a small affine
subset. We have a natural isomorphism
 $
\HH'\otimes_{\Z_p}B(R)\cong H'\otimes_{R} B(R)
 $
respecting $\Gamma_R$-actions and filtrations (we forget the
$\phi$s in the isomorphism). As $\Q_{p^d}\subset B(R)$, we can
write it as
$$
(\tilde \V_1\otimes\cdots\otimes
\tilde\V_d)\otimes_{\Q_{p^d}}B(R)\cong(H_1\otimes \cdots\otimes
H_d)\otimes_{R[\frac{1}{p}]}B(R),
$$
or
$$
\bigotimes_{i=1}^{d}[\tilde
\V_i\otimes_{\Q_{p^d}}B(R)]\cong\bigotimes_{i=1}^{d}[H_i\otimes_{R[\frac{1}{p}]}B(R)].
$$
In the comparison the tensor factor with numbering is preserved,
because it is so over a general $\bar \Q$-rational point of each
connected component of $U$ by a result of Blasius and Wintenberger
(see \cite{Bl}, see also \S4 \cite{Og}), which asserts that in the
$p$-adic comparison the tensors $s_{\alpha,et}$ and $s_{\alpha,dR}$
correspond. Then taking the second wedge (symmetric) power for $n$
even (odd) of the above isomorphism, we find the isomorphism $
\Sym^2\V_1\otimes_{\Q_{p^d}}B(R)\cong\Sym^2H_1\otimes_{R[\frac{1}{p}]}B(R)
 $
which respects $\Gamma_R$-actions and filtrations. Taking
$\Gamma_R$-invariants of both sides, we obtains the claimed
isomorphism over $\Spec R[\frac{1}{p}]$. By the naturalness of the
comparison, the local isomorphisms glue into a global one.
\end{proof}

By the main theorems of Langton \cite{La}, the line bundle
$\sL^{-1}$ extends over $M\otimes \bar \Z_p$ with the modulo $p$
reduction $\sL_0^{-1}$, and the isomorphism in the above lemma
specializes to an isomorphism between $\sP_0$ and $\sL_0^{-2}$. So
we have shown the first isomorphism in the following
\begin{proposition}\label{grading}
One has natural isomorphisms
 $
\sP_0\cong\sL_0^{-2}\cong\Omega_{M_0}^{-1}
 $ over $\bar M_0$.
\end{proposition}
\begin{proof}
We have shown that over $\bar M^0$ the Higgs field $\theta_1$
induces an isomorphism $\sL^2\cong\Omega_{M^0}$. For the same
reason as above, this isomorphism specializes into an isomorphism
$\sL_0^2\cong\Omega_{M_0}$.
\end{proof}

{\itshape Construction of $\tilde F_{rel}$.} For each small affine
$U\in \sU$, we choose a Frobenius lifting $F_U: \hat U\to \hat U$,
where $\hat U$ is the $p$-adic completion of $U$. As $\EE_0$ is an
object in $\mathcal{MF}^{\nabla}_{big,r}(M)$, there is a map
 $
\phi_{r,F_U}: F_U^{*r}\EE_0|_{\hat U}\to \EE_0|_{\hat U}
 $.
By Proposition \ref{direct sum decomposition of second/wedge power},
$\phi_{r,F_U}$ is the restriction of the second wedge (resp.
symmetric) power of the $r$-th iterated relative Frobenius morphism
$\phi_{F_U}: F_U^{*}H'|_{\hat U}\to H'|_{\hat U}$ for $d$ even
(resp. odd) to the direct factor $\EE_0|_{\hat U}$.
\begin{lemma}\label{jumping is controled by one factor}
For each $U$, the image $\phi_{r,F_U}(F_U^{*r}\EE_0|_{\hat
U})\subset \EE_0|_{\hat U}$ is divisible by $p^{r-1}$, but not
divisible by $p^r$.
\end{lemma}
\begin{proof}
In the $p$-adic filtration $$\EE_0|_{\hat U}\supset p\EE_0|_{\hat
U}\cdots \supset p^{i-1}\EE_0|_{\hat U}\supset p^{i}\EE_0|_{\hat
U}\supset \cdots,$$ there is a unique $i$ with the property
$$ p^{i-1}\EE_0|_{\hat U}\supset \phi_{r,F_U}(F_U^{*r}\EE_0|_{\hat U}) \supsetneq p^{i}\EE_0|_{\hat U}
.$$ It is to show that $i=r$, or equivalently that the images of
$\phi_{r,F_U}(F_U^{*r}\EE_0|_{\hat U})$ in the successive gradings
$\frac{p^{i-1}\EE_0|_{\hat U}}{p^{i}\EE_0|_{\hat U}}$ are zero for
$1\leq i<r$ and nonzero for $i=r$. Let $x_0\in \hat U(k)$ and $\hat
U_{x_0}$ the completion of $\hat U$ at $x_0$. It is equivalent to
show the above statement over each $\hat U_{x_0}$. This follows from
the description of the relative Frobenius $\phi$ over the formal
neighborhood $\hat U_{x_0}$ as described in \S\ref{subsection on
formal tensor decomposition}  and the result for the closed point
$x_0$. In detail it goes as follows: by Proposition \ref{application
of global to local}, the filtered $\phi^r$-module $\EE_0|_x$ is
isomorphic to $\Sym^2M_0\otimes (\bigotimes_{i=1}^{r-1}\det
M_i)\otimes \textrm{unit crystal}$ over $W(\bar k)$. By Proposition
\ref{possible newton slopes}, the Newton slope of the rank one
$\phi^r$-module $\det M_i, i\geq 1$ is either $1\times 1$ or
$1\times 2$. In the former case, the Newton slopes of $\Sym^2M_0$
are $\{1\times 0, 1\times 1, 1\times 2\}$. These imply that
$\phi_r(\EE_0|_x)$ is always divisible by $p^{r-1}$. By Remark
\ref{remark on the existence of NP}, the former case does occur for
a certain $x_0$. So $\phi_r(\EE_0|_x)$ is not divisible by $p^r$ at
such a closed point.
\end{proof}
As $\phi_{r,F_U}(F_U^{*r}F^1\EE_0|_{\hat U})\subset
p^{r}\EE_0|_{\hat U}$, the composite of the following morphisms
$$
F_U^{*r}F^1\EE_0|_{\hat U}\hookrightarrow F_U^{*r}\EE_0|_{\hat U}
\stackrel{\frac{\phi_{r,F_{U}}}{p^{r-1}}}{\longrightarrow}
\EE_0|_{\hat U}\stackrel{pr}{\twoheadrightarrow} \sP|_{\hat
U}\stackrel{\mod p}{\twoheadrightarrow} \sP_0|_{U_0}
$$
is zero. As a result it gives the following morphism
$$
\frac{F_U^{*r}\EE_0|_{\hat U}}{F_U^{*r}F^1\EE_0|_{\hat
U}}=F_U^{*r}\sP|_{\hat U}\to \sP_0|_{U_0},
$$
which clearly factors further through $F_{U}^{*r}\sP|_{\hat
U}\stackrel{\mod p}{\twoheadrightarrow} F_{U_0}^{*r}\sP_0|_{U_0}$.
Thus we obtain a morphism $F_{U_0}^{*r}\sP_0|_{U_0}\to
\sP_0|_{U_0}$ which is denoted by
$[\frac{\phi_{r,F_{U}}}{p^{r-1}}]$.
\begin{lemma}\label{local relative frobenius glues modulo p}
The local morphisms $\{[\frac{\phi_{r,F_{U}}}{p^{r-1}}]\}_{U\in
\sU}$ glue into a global one $ \tilde F_{rel}: F_{M_0}^{*r}\sP_0\to
\sP_0
 $.
\end{lemma}
\begin{proof}
It is equivalent to show the following statement: for two
different Frobenius liftings $F_U$ and $F'_U$ of the absolute
Frobenius $F_{U_0}$, and for a local section of
$(F_{M_0}^{*r}\sP_0)(U_0)$ of form $F_{U_0}^{*r}s_0$ with $s_0\in
\sP_0(U_0)$, one has the equality
 $
[\frac{\phi_{r,F_{U}}}{p^{r-1}}](F_{U_0}^{*r}s_0)=[\frac{\phi_{r,F'_{U}}}{p^{r-1}}](F_{U_0}^{*r}s_0)
 $.
Let $s$ be an element of $\EE_0(\hat U)$ lifting $s_0$. It is to
show that
 $
(\frac{\phi_{r,F'_{U}}}{p^{r-1}}F'^{*r}_U-\frac{\phi_{r,F_{U}}}{p^{r-1}}F_U^{*r})(s)\in
p\EE_0(\hat U)
 $.
Note that by replacing the Frobenius $\phi_r$ of $\EE_0$ with
$\frac{\phi_r}{p^{r-1}}$ one obtains another object in
$\mathcal{MF}^{\nabla}_{big,r}(M)$, which is denoted by $\EE_0'$.
Let $x_0\in \hat U(k)$ and $\hat U_{x_0}$ be the completion of $U$
at $x_0$. Fix an isomorphism $\hat U_{x_0}\cong W(k)[[t]]$. Then
$F_U$ and $F'_U$ restrict to two Frobenius lifting on $\hat
U_{x_0}$. For any local section $s'$ of $\EE_0'(\hat U_{x_0})$, one
has the Taylor formula (see \S7 \cite{Fa3}, Theorem 2.3 \cite{Fa2},
page 16 \cite{Ki}): write $\partial=\partial_t$ and
$z=F'_U(t)-F_U(t)$. Then it holds that
$$
\frac{\phi_{r,F'_{U}}}{p^{r-1}}F'^*_U(s')=\sum_{i=0}^{\infty}\frac{\phi_{r,F_{U}}}{p^{r-1}}F^*_U(\nabla_{\partial}^i(s'))\otimes
\frac{z^i}{i!}.
$$
Note that as $z$ is divisible by $p$, $\frac{z^i}{i!}$ is
divisible by $p$ for all $i\geq 1$. So the difference
$\frac{\phi_{r,F'_{U}}}{p^{r-1}}F'^*_U(s')-\frac{\phi_{r,F_{U}}}{p^{r-1}}F^*_U(s')$
belongs to $p\EE_0'(\hat U_{x_0})$. The lemma follows.
\end{proof}

Let $\sS\subset M_0(\bar k)$ be the supersingular locus of
$f_0:X_0\to M_0$.
\begin{proposition}\label{supersingular locus coincides with frobenius degeneracy
locus} The morphism $\tilde F_{rel}$ is nonzero and takes zero at
$x_0\in M_0(\bar k)$ iff $x_0\in \sS$.
\end{proposition}
\begin{proof}
The morphism $\tilde F_{rel}$ is nonzero because of Lemma
\ref{jumping is controled by one factor}. And when and only when it
takes zero at $x_0$, the Newton slopes of the factors $M_i$ in the
proof of Lemma \ref{jumping is controled by one factor} take value
in $\{2\times 1\}$, which by the proof of Theorem
\ref{classification and existence of newton polygon} implies that
$x_0\in \sS$.
\end{proof}

\subsection{A mass formula}
In this paragraph we deduce a mass formula for the supersingular
locus $\sS$ from the pair $(\sP_0,\tilde F_{rel})$. It is clear that
we shall determine the multiplicity of the Frobenius degeneracy at a
supersingular point. To that we have the following result:
\begin{proposition}\label{mutiplicity two property}
The vanishing order of $\tilde F_{rel}$ at each supersingular point
is two.
\end{proposition}
This is a local statement. Take an $x_0\in\sS\cap M_0(k)$. By
discussions in \S\ref{Drinfel'd module and versal deformation
subsection}, there is a Drinfel'd $\sO_{\mathfrak{p}}$-divisible
module $B'$ such that Corollary \ref{claim 3} holds. It is also
clear that $B'$ is supersingular. In this case, it is a formal
$p$-divisible group. By Proposition \ref{application of global to
local}, the above statement can be deduced from the corresponding
result for the universal filtered Dieudonn\'{e} module associated to
a versal deformation of a Drinfel'd $\sO_{\mathfrak p}$-divisible
module. To this end we shall apply the theory of display for a local
expression of the Frobenius. Note that $\Sym^2\sN_0\otimes
\bigotimes_{i=1}^{r-1}\det(\sN_i)$ is contained as a direct factor
in $\bigwedge^2(\otimes_{i=0}^{r}\sN_i)$ (resp.
$\Sym^2(\otimes_{i=0}^{r}\sN_i)$) for $r$ even (resp. odd). The
induced Frobenius on the factor $\Sym^2\sN_0\otimes
\bigotimes_{i=1}^{r-1}\det(\sN_i)$ from the second wedge/symmetric
power of $\phi^r_{ten}$ on $\otimes_{i=0}^{r}\sN_i$ is denoted by
${\phi^r_{ten}}^{\otimes 2}$. We have then the following
\begin{proposition}\label{vanishing order one property}
The vanishing order of $\phi_{\sN_0} \mod p$ on
$\frac{\sN_0}{Fil^1\sN_0}$ along the equal characteristic
deformation at the point $[B']$ is one, and that of
$\frac{{\phi^r_{ten}}^{\otimes 2}}{p^{r-1}} \mod p$ on
$\frac{\Sym^2\sN_0}{Fil^1\Sym^2\sN_0}\otimes
\bigotimes_{i=1}^{r-1}\det \sN_i$ is two.
\end{proposition}
\begin{proof}
Note that it suffices to write down the display over the
equal-characteristic deformation (see \cite{No}, \cite{NO}, \S2
\cite{GO}). For simplicity we shall take $r=2$ in the following
argument. The proof for a general $r$ is completely the same. Let
$(N,F,V)$ be the covariant Dieudonn\'{e} module of the Cartier dual
of the Drinfel'd $\sO_{\mathfrak p}$-divisible module $B'$ over
$\bar k$. So we have the eigen decomposition $N=N_0\oplus N_1$ with
respect to the endomorphism $\sO_{\mathfrak p}\cong\Z_{p^2}$. Choose
a basis $\{X_i,Y_i\}$ for $N_i, i=0,1$. To write down the display,
we need to arrange the order of the basis elements into
$\{Y_0,X_1,Y_1,X_0\}$ with the understanding that $X_0$ modulo $p$
is the basis element of $\frac{VN}{pN}$ which is one dimensional
$\bar k$-vector space. Then the display under the chosen basis is
given by the matrix:
$$
\left(
  \begin{array}{cc}
    A_{3\times 3} & B_{3\times 1} \\
    C_{1\times 3} & D_{1\times 1} \\
  \end{array}
\right)= \left(
  \begin{array}{cccc}
    0 & c_1 & d_1 & 0 \\
    b_1 & 0 & 0 & a_1 \\
    b_2 & 0 & 0 & a_2 \\
    0 & c_2 & d_2 & 0 \\
  \end{array}
\right)
$$
This is an invertible matrix, i.e.  $\det\left(
                                             \begin{array}{cc}
                                               a_1 & b_1 \\
                                               a_2 & b_2 \\
                                             \end{array}
                                           \right)\cdot \det\left(
                                             \begin{array}{cc}
                                               c_1 & d_1 \\
                                               c_2 & d_2 \\
                                             \end{array}
                                           \right)
 $ is a unit. Since both determinants are elements in $W(\bar k)$, it implies that each determinant is a unit in
$W(\bar k)$. The universal equal-characteristic deformation ring of
$B'$ as $p$-divisible group is $\bar k[[t_0,t_1,t_2]]$. Let $T_i\in
W(\bar k[[t_0,t_1,t_2]])$ be the Teichm\"{u}ller lifting of $t_i$
for $0\leq i\leq 2$. Then by Norman and Norman-Oort loc. cit., the
display over the universal equal-characteristic deformation is given
by $ \left(
  \begin{array}{cc}
    A+TC & B+TD \\
    C & D \\
  \end{array}
\right)
 $, where $T= \left(%
\begin{array}{ccc}
  T_0 & T_1 & T_2 \\
\end{array}%
\right)^{t}$. And the Frobenius on the universal display is given
by
$$
M_1:=\left(
  \begin{array}{cc}
    A+TC & p(B+TD) \\
    C & pD \\
  \end{array}
\right).
$$
We need to determine the one dimensional sublocus of $\Spf(\bar
k[[t_0,t_1,t_2]])$ where $B'$ deforms as a Drinfel'd module. Take
$s\in \Z_{p^2}$ to be a primitive element. Then the endomorphism
of $N$ given by $s$ has the matrix form (using the same basis):
$$
M_2:=\left(
  \begin{array}{cccc}
    \xi & 0&0&0 \\
    0 & \xi^{\sigma}&0&0 \\
 0 & 0&\xi^{\sigma}&0 \\
  0 &0&0& \xi \\
  \end{array}
\right).
$$
The universal display of the Drinfel'd module has the property
that the endomorphism matrix commutes with the Frobenius. That is
one has
 $
M_1M_2^{\sigma}=M_2M_1
 $.
Now by an easy computation one finds that the one dimensional
deformation as the Drinfel'd module is given by $t_1=t_2=0$. Write
$t=t_0$. Thus the two-iterated Frobenius $\phi_{\sN_{B'}}^2$ on
$\sN_{B'}$ along the equal-characteristic deformation is displayed
by
$$
\phi_{\sN_{B'}}^2\{Y_0,X_1,Y_1,X_0\}=\{Y_0,X_1,Y_1,X_0\}\Phi,
$$
where $\Phi=M_1M_1^{\sigma}$ is equal to
$$
\left(
  \begin{array}{cccc}
  \Phi_{11}  & 0 & 0 &  \Phi_{14} \\
    0 &   \Phi_{22}  &  \Phi_{23}  & 0 \\
    0 &   \Phi_{32} &  \Phi_{33} & 0 \\
      \Phi_{41} & 0 & 0 &   \Phi_{44} \\
  \end{array}
\right).
$$
The nontrivial entries are given by
\begin{eqnarray*}
  \Phi_{11}&=& (b_1^{\sigma}c_1+b_2^{\sigma}d_1)+(b_1^{\sigma}c_2+b_2^{\sigma}d_2)t,  \Phi_{14}=(pa_1^{\sigma}c_1+pa_2^{\sigma}d_1)+(pa_1^{\sigma}c_2+pa_2^{\sigma}d_2)t,\\
  \Phi_{22} &=& (b_1c_1^{\sigma}+pa_1c_2^{\sigma})+b_1c_2^{\sigma}t^{\sigma},\quad \quad\quad\Phi_{23} =(b_1d_1^{\sigma}+pa_1d_2^{\sigma})+b_1d_2^{\sigma}t^{\sigma}, \\
 \Phi_{32} &=& (b_2c_1^{\sigma}+pa_2c_2^{\sigma})+b_2c_2^{\sigma}t^{\sigma},\quad \quad\quad   \Phi_{33}=(b_2d_1^{\sigma}+pa_2d_2^{\sigma})+b_2d_2^{\sigma}t^{\sigma},\\
\Phi_{41}&=&b_1^{\sigma}c_2+b_2^{\sigma}d_2,\quad\quad
\quad\quad\quad\quad \quad\quad \ \Phi_{44}
=pa_1^{\sigma}c_2+pa_2^{\sigma}d_2.
\end{eqnarray*}
Consider first the element $\Phi_{11}$: its modulo $p$ reduction is
equal to the iterated Hasse-Witt map on $\frac{\sN_0}{Fil^1\sN_0}$.
As we require that $B'$ lies in the supersingular locus which is a
finite set, it follows that
$$
b_1^{\sigma}c_1+b_2^{\sigma}d_1= 0 \mod p,
b_1^{\sigma}c_2+b_2^{\sigma}d_2 \neq 0 \mod p.
$$
This shows the first assertion in the statement. So we can write
that $b_1^{\sigma}c_1+b_2^{\sigma}d_1=pv_1,
b_1^{\sigma}c_2+b_2^{\sigma}d_2=u_1$, where $u_1$ is a unit.
Consider the induced Frobenius on $\Sym^2\sN_0\otimes
\bigwedge^2\sN_1$. We shall compute the coefficient before the
element $Y_0^2\otimes X_1\wedge Y_1$, which is the basis element
of $\frac{\Sym^2\sN_0}{Fil^1\Sym^2\sN_0}\otimes \det \sN_1$, under
the map $\frac{{\phi_{ten}^2}^{\otimes 2}}{p}\mod p$. Using the
above matrix expression of $\phi_{\sN_{B'}}^2$ one computes that
the local expression is given by
$$
[-u_1^2\det\left(
             \begin{array}{cc}
               a_1 & b_1 \\
               a_2 & b_2 \\
             \end{array}
           \right)\det\left(
             \begin{array}{cc}
               c_1 & d_1 \\
               c_2 & d_2 \\
             \end{array}
           \right)^{\sigma}]
t^2+v_2t^p+v_3t^{p^2}.
$$
By the previous discussion we know that the coefficient before
$t^2$ is a unit. As $p$ is assumed to be odd, it follows that the
multiplicity is equal to two.
\end{proof}
Now the proof of Proposition \ref{mutiplicity two property} is
clear:
\begin{proof}
By the construction of $\tilde F_{rel}$, its vanishing order at
$x_0$ is equal to that of $\frac{\phi_r}{p^{r-1}} \mod p$ on
$\frac{{\EE_0}}{F^1{\EE_0}}$ along $\hat M_{0,x_0}$. Note that the
closed formal subscheme $\hat M_{0,x_0}\subset \hat M_{x_0}$
represents the equal-characteristic deformation direction. By
Proposition \ref{application of global to local}, the restriction of
${\EE_0}$ to $\hat M_{x_0}$ is naturally isomorphic to $\hat
\xi_{crys}[\Sym^2\sN_0\otimes \bigotimes_{i=1}^{r-1}\det(\sN_i)]$.
Thus the result follows from Proposition \ref{vanishing order one
property}.
\end{proof}

\begin{corollary}\label{Newton jumping locus}
Let $\sS$ be the supersingular locus of $f_0: X_0\to M_0$. Then in
the Chow ring of $\bar M_0$ one has the cycle formula
$$
2\sS=(1-p^r)c_1(M_0),
$$
where $r=[F_{\mathfrak p}:\Q_p]$. Consequently one has the
following mass formula
$$
|\sS|=(p^r-1)(g-1),
$$
where $g$ is the genus of $M_0$.
\end{corollary}
\begin{proof}
By Propositions \ref{supersingular locus coincides with frobenius
degeneracy locus} and \ref{mutiplicity two property}, it follows
that
$$
2\sS=(p^r-1)c_1(\sP_0).
$$
By Proposition \ref{grading}, one has further
$$
c_1(\sP_0)=-2c_1(\sL_0)=-c_1(M_0).
$$
By taking the degree of the cycle formula, one obtains the mass
formula as claimed.
\end{proof}

\end{document}